\documentclass[twoside,reqno]{amsart} 

\usepackage{xcolor}
\usepackage{amssymb}
\usepackage{tikz}

\newtheorem{theorem}{Theorem}[section]
\newtheorem{lemma}[theorem]{Lemma}
\newtheorem{corollary}[theorem]{Corollary}
\newtheorem{proposition}[theorem]{Proposition}

\theoremstyle{definition}
\newtheorem{definition}[theorem]{Definition}
\newtheorem{remark}[theorem]{Remark}
\newtheorem{example}[theorem]{Example}

\numberwithin{equation}{section}

\DeclareMathOperator{\wind}{wind}
\DeclareMathOperator{\ind}{ind}

\DeclareMathOperator{\diag}{diag}
\DeclareMathOperator{\im}{im}
\DeclareMathOperator{\tr}{trace}

\renewcommand{\Re}{\mbox{\rm Re\,}}

\newcommand{\R}{{\mathbb R}}
\newcommand{\C}{{\mathbb C}}
\newcommand{\Z}{{\mathbb Z}}
\newcommand{\N}{{\mathbb N}}
\newcommand{\T}{{\mathbb T}}

\newcommand{\be}{\begin{equation}}
\newcommand{\ee}{\end{equation}}
\newcommand{\nn}{\nonumber}
\newcommand{\ba}{\begin{array}}
\newcommand{\ea}{\end{array}}

\newcommand{\wh}{\widehat}
\newcommand{\wt}[1]{\widetilde{#1}}
\newcommand{\iv}{^{-1}}
\newcommand{\iy}{\infty}
\newcommand{\NN}{^{N\times N}}

\renewcommand{\kappa}{\varkappa}
\newcommand{\eps}{{\varepsilon}}
\newcommand{\Ga}{\Gamma}

\newcommand{\la}{\lambda}

\newcommand{\cB}{\mathcal{B}}
\newcommand{\cG}{\mathcal{G}}

\newcommand{\cL}{\mathcal{L}}
\newcommand{\cM}{\mathcal{M}}

\newcommand{\ta}{\tilde{a}}
\newcommand{\tb}{\tilde{b}}

\newcommand{\tp}{\tilde{\phi}}

\newcommand{\twotwo}[4]{\left(\begin{array}{cc}#1&#2\\[1ex]#3&#4\end{array}\right)}

\newcommand{\CpwTG}{C^{1+\eps}_{\mathrm{pw}}(\T;\Gamma) }

\title[Block Toeplitz determinants with piecewise continuous functions]{Asymptotics of block Toeplitz determinants with piecewise continuous symbols}
\author{E. Basor}
\address{American Institute of Mathematics, San Jose, USA}
\email{ebasor@aimath.org}
\author{T. Ehrhardt}
\address{University of California, Santa Cruz, USA}
\email{tehrhard@ucsc.edu}
\author{J. A. Virtanen}
\address{University of Reading, Reading, UK}
\email{j.a.virtanen@reading.ac.uk}

\begin{document}

\begin{abstract}
We determine the asymptotics of the block Toeplitz determinants $\det T_n(\phi)$ as $n\to\infty$ for $N\times N$ matrix-valued piecewise continuous functions $\phi$ with a finitely many jumps under mild additional conditions. In particular, we prove that
$$
	\det T_n(\phi) \sim G^n n^\Omega E\quad {\rm as}\ n\to \infty,
$$
where $G$, $E$, and $\Omega$ are constants that depend on the matrix symbol $\phi$ and are described in our main results. Our approach is based on a new localization theorem for Toeplitz determinants, a new method of computing the Fredholm index of Toeplitz operators with piecewise continuous matrix-valued symbols, and other operator theoretic methods. As an application of our results, we consider piecewise continuous symbols that arise in the study of entanglement entropy in quantum spin chain models.
\end{abstract}

\maketitle
\setcounter{tocdepth}{1}
\tableofcontents

\section{Introduction}
For a matrix-valued  function $\phi \in L^\infty(\T)^{N\times N}$ defined on the unit circle $\T=\{z\in\C\,:\,|z|=1\}$ with Fourier coefficients $\phi_{j}$, define the Toeplitz determinants $D_n[\phi]$ of the finite block Toeplitz matrix $T_n(\phi)$ by
$$
	D_n[\phi] = \det T_n(\phi) = \det (\phi_{j-k})_{j,k=0}^{n-1},  \qquad n\in \N.
$$
The asymptotic behavior of $D_n[\phi]$ as $n\to\infty$ is well understood when $\phi$ is sufficiently nice and it is given by the well-known Szeg\H{o}-Widom limit theorem, see~\cite{BS06, Wi76}. For symbols possessing zeros, certain kinds of singularities, jump discontinuities, or having a nonzero winding number, the large $n$ behavior of Toeplitz determinants is nearly completely understood only when $N=1$ and given by the Fisher-Hartwig asymptotics. For details we refer to \cite{DIK11, E01} and also to \cite{BS99,DIK13} for more general information.

This paper is concerned with the asymptotic behavior of block Toeplitz determinants with Fisher-Hartwig symbols. Specifically we deal with the case of symbols with jump discontinuities. Our approach is based in part on the localization or separation theorem \cite{B1979} which states that when the symbols $\phi$ and $\psi$ do not have common singularities and satisfy certain invertibility and smoothness criteria off the singularites, then
\begin{equation*}\label{e:B}
	\lim_{n\to \infty} \frac{D_n[\phi \psi]}{D_n[\phi] D_n[\psi]}
	=\det\Big( T^{-1}(\phi) T({\phi\psi}) T^{-1}(\psi) \Big)
	 \Big( T^{-1}({\tilde{\phi} })T({ \widetilde {\phi \psi}}) T^{-1}({\tilde{\psi} })\Big).
\end{equation*}
In the above, $T(\phi)$ is the semi-infinite Toeplitz operator defined on $\ell^{2}(\Z_+)^{N}$, $\Z_+=\{0,1,\dots\}$ with matrix entry $\phi_{j-k}$ and $\tilde{\phi}(e^{i\theta}) = \phi(e^{-i \theta}).$
The localization theorem proved useful because if one could find a canonical symbol that possessed one jump singularity and such that the determinant asymptotics were known for the canonical symbol, then the asymptotics could be constructed for an arbitrary symbol with a finite number of jumps by applying it to a pair of symbols with disjoint singularities and then by repeatedly adding another canonical factor. 

It might seem that this idea should easily transform to the matrix-valued symbol  case. However, the localization theorem requires at each step that certain semi-infinite Toeplitz operators be invertible. In the scalar case this is not an issue. This is because in the scalar case, if two invertible Toeplitz operators have bounded symbols that have disjoint singularities, then the Toeplitz operator with the product symbol is also invertible. However in the block case, one can only say that the resulting operator is Fredholm with index zero.

Thus a new version of the localization theorem needs to be proved that does not require the same invertibility conditions. This is what will be done in this paper. With the new version and under appropriate conditions on $\phi$ we prove that 
\begin{equation}\label{Dn.asym}
  D_{n}[\phi] \sim G^{n} \, n^{\Omega} E,\qquad\mbox{ as } n\to\infty, 
\end{equation}
where $G$, $E$, and $\Omega$ are constants that depend on $\phi$ and can be described.
Our main results require some preparations and will be stated Section \ref{sec:main-results}, where we also provide some comments on the constant $E$. Some auxiliary results and their proofs, definitions such as $I$-regularity and $I$-winding number, and operator theoretic preliminaries will be given in Section \ref{preliminaries}.
The proofs of the main results are given in Sections \ref{1st results} and \ref{final results}, which is followed by a brief discussion of a possible alternate approach based on Widom's perturbation result \cite{Wi75a} and some open problems in Section~\ref{alt-approach}. Examples that illustrate our results will be given in Section \ref{sec:Examples}.

In order to put our result into context, let us first recall a version of the Szeg\H{o}-Widom theorem. 
Therein $F=W\cap F\ell^{2,2}_{1/2,1/2}$ stands for the set of all functions $a\in L^1(\T)$ with  Fourier coefficients $a_n$ satisfying
$$
\|a\|_{F}:=\sum_{n=-\infty}^\infty |a_n| +  \left(\sum_{n=-\infty}^\infty  |n| \cdot |a_n|^2\right)^{1/2}<\infty.
$$

\begin{theorem}[Szeg\H{o}-Widom]\label{Sz.W}
Let $\phi\in F\NN$ be such that the determinant $\det \phi(t)$ does not vanish on all of $\T$ and has winding number zero. 
Then 
$$
\lim_{n\to\infty} \frac{\det T_n(\phi)}{G[\phi]^n}=\det T(\phi)T(\phi\iv),
$$
where the right hand side is a well-defined operator determinant and
\be\label{const.G}
G[\phi] =\exp\left( \frac{1}{2\pi} \int_{0}^{2\pi} (\log \det \phi)(e^{ix}) \, dx\right)
\ee
in which $\log \det \phi$ is continuous on $\T$.
\end{theorem}

Besides the original references of \cite{Wi74,Wi76, Wi75b} a slightly different operator-theoretic proof can be found in \cite[Sect.~10.25-32]{BS06}. We remark that if both $T(\phi)$ and $T(\phi^{-1})$ are invertible, the proof is easier than in the general case where the stated assumption on $\det \phi(t)$ is equivalent to both $T(\phi)$ and $T(\phi^{-1})$ being Fredholm operators with index zero (see also Theorem  \ref{thm:index} below). Notice also that under the stronger assumption the Szeg\H{o}-Widom theorem follows immediately from the Geronimo-Case-Borodin-Okounkov formula (see \cite[Sect.~10.40]{BS06} and the references therein).
Another proof based on a different approach which uses Banach algebras is given in \cite{E03}. We remark that the
class $F$ considered above can be replaced by more general classes such as Krein algebras. Furthermore,
in the scalar case ($N=1$) a multitude of different proofs of the classical Szeg\H{o} Limit Theorem exist.

Let us now briefly recall what is known about the asymptotics of the determinants $\det T_n(\phi)$ 
for scalar ($N=1$) symbols $\phi$ with jump discontinuities. We assume that the symbol is represented as a product
\begin{equation}\label{f.prod_sc}
\phi(t)= \phi_0(t) \prod_{k=1}^R u_{\beta_k,\tau_k}(t)
\end{equation}
where $\phi_0$ is a sufficiently smooth nonvanishing function on $\T$ with winding number zero, and the functions $u_{\beta,\tau}$
having a single jump at $t=\tau$ are defined by
\begin{equation}\label{u-beta}
u_{\beta,\tau}(t)=(-t/\tau)^{\beta}=\exp(i\beta  \arg(-t/\tau)),\qquad t\in\T,
\end{equation}
with $|\!\arg(\,\cdot\,)|<\pi$.
The numbers $\tau_1,\dots,\tau_R\in\T$ are distinct, and $\beta_1,\dots,\beta_R\in\C$ are the jump parameters.
The Fisher-Hartwig type asymptotics for the determinants $\det T_n(\phi)$ is given by \eqref{Dn.asym} with the constant
$G=G[\phi_0]$ defined in \eqref{const.G}, 
$$
\Omega=-\sum_{k=1}^R \beta_k^2
$$
and a more complicated but explicit constant $E\neq0$. More specifically, in the case of multiple jumps these asymptotics were first proved  \cite{B1978} under the assumption
\begin{itemize}
\item[(a)] $\Re\beta_k=0$ for all $1\le k\le R$.
\end{itemize}
This condition was soon replaced \cite{B1979,Bo82,Blek} by the weaker assumption
\begin{itemize}
\item[(b)] $|\Re\beta_k|<1/2$ for all $1\le k\le R$.
\end{itemize}
Finally, it was proved  \cite{E01} that the asymptotics are valid even under the condition
\begin{itemize}
\item[(c)] $|\Re\beta_k-\Re\beta_j|<1$  and $\beta_k\notin\Z\setminus\{0\}$ for all $1\le j,k\le R$.
\end{itemize}
This last condition on the parameters is sharp. Indeed, if merely $|\Re\beta_k-\Re\beta_j|\le 1$ is assumed and equality is attained for at least some $j,k$, then the Fisher-Hartwig asymptotics breaks down and a generalized asymptotic formula has been proved \cite{DIK11}. Let us also remark that before the general case (c) was established, the following modifications of (b),
\begin{itemize}
\item[(b1)] $0\le \Re\beta_k<1$ for all $1\le k\le R$,\\[-1.3ex]
\item[(b2)] $-1< \Re\beta_k \le 0$ for all $1\le k\le R$,
\end{itemize}
have been dealt with by similar techniques \cite{BM}.

Our main results concern the block case of Fisher-Hartwig symbols with jump discontinuities. The assumptions we need to impose correspond in the scalar case to condition (b) above. Therefore, while this covers a broad situation, it is not the most general case for which the results can be expected to hold. Perhaps cases corresponding to (b1) or (b2) can be established by  
slightly modifying our method, but what corresponds to case (c) or 
to the generalized Fisher-Hartwig asymptotics (aka the Basor-Tracy asymptotics) is considerably more challenging and seems currently out of reach.

As far as the authors are aware of, no general results for Fisher-Hartwig type symbols in the block case are known up to now.
It is possible that for very specific block symbols some results have been obtained  in the literature. For instance, the work of \cite{AEFQ1, AEFQ2} to be
discussed below contains non-rigorous results for particular block jump symbols. There are some cases that be can be trivially reduced to the scalar case, e.g., if the symbol can be transformed into to block triangular matrix functions by (left/right) multiplication with nonsingular constant matrices. Let us also note that for piecewise continuous symbols, it is quite obvious what kind of symbols are their generalization to the block case. However, it is less clear what should constitute the block analogue of general Fisher-Hartwig type symbols or even symbols that only involve zero/pole-type singularities.

\subsection{Application: Entanglement entropy} In many instances entanglement entropy of various quantum spin chain models, such as the XX, XY and Ising chains, can be computed using the Szeg\H{o}-Widom limit theorem or determinants involving Toeplitz matrices generated by $2\times 2$ matrix-valued symbols that possess jump discontinuities. The former, when the smooth symbol is matrix-valued, still requires the computation of the constant in the expansion that is known only in rare cases, such as those of Its, Mezzadri and Mo~\cite{IMM}, in which the authors compute the von Neumann entropy of entanglement of the ground state of a wide family of one-dimensional quantum spin chain models (incl.~the XX and the XY models).

Jin and Korepin \cite{JK} were the first to rigorously compute the von Neumann entropy of the ground state of the XX model, and in particular showed that the entropy grows like $\frac13\log L$ (where $L$ is the length of the chain) at a phase transition using the asymptotics of Toeplitz determinants with piecewise continuous \emph{scalar}-valued symbols. Toeplitz determinants with non-singular \emph{matrix}-valued symbols first appeared in the computation of the entropy of the XY model in~\cite{IJK} and other more general one-dimensional models in~\cite{IMM}.

The basic idea of how Toeplitz determinants enter the study of entanglement is as follows. Consider the Hamiltonian
\begin{multline}\label{e:H}
	H_\alpha = -\frac\alpha2 \sum_{0\le j\le k\le M-1}
	\left( (A_{jk} + \gamma B_{jk})\sigma^x_j \sigma^x_k
	\prod_{l=j+1}^{k-1} \sigma_l^z\right. \\
	\left. + (A_{jk} - \gamma B_{jk})\sigma^y_j \sigma^y_k
	\prod_{l=j+1}^{k-1} \sigma_l^z  \right) - \sum_{j=0}^{M-1} \sigma^z_j,
\end{multline}
where $\sigma^x_j, \sigma^y_j, \sigma^z_j$ stand for the Pauli matrices which describe spin operators on the $j$th lattice site of a chain with $M$ 
sites, $A$ is symmetric, $B$ is antisymmetric, and both are translation-invariant. We note that this generalizes the XY model whose Hamiltonian is given by
\begin{equation}\label{e:XY}
	H_{\alpha}^{\mathrm{XY}} = -\frac\alpha2 \sum_{j=0}^{M-1}
	\left( (1+\gamma)\sigma^x_j\sigma^x_{j+1} + (1-\gamma)\sigma^y_j\sigma^y_{j+1}\right) - \sum_{j=0}^{M-1} \sigma^z_j,
\end{equation}
where $\gamma\in [0,1]$. Further, when $\gamma=0$, \eqref{e:XY} provides the Hamiltonian of the XX model. We also remark that at the critical value $\alpha=1$, the XY model undergoes a phase transition. Going back to the Hamiltonian in \eqref{e:H}, if we divide the system into two subchains, denoting the part containing the first $L$ spins by $A$ and the second part containing the remaining $M-L$ spins by $B$ with $1\ll L \ll M$, then the von Neumann entropy $S(\rho_A)$ is given by
\begin{equation}\label{e:entropy}
	S(\rho_A) = -\tr \rho_A \log \rho_A,
\end{equation}
where $\rho_A = \tr_B \rho_{AB}$ and $\rho_{AB} = |\Psi_g \rangle \langle \Psi_g|$. It turns out that (see, e.g., \cite{IMM})
\begin{equation}\label{e:entropy-rep}
	S(\rho_A) = \lim_{\epsilon \to 0} \frac1{4\pi i} \int_{\Gamma(\epsilon)}
	e(1+\epsilon, \lambda) \frac{d \log D_L(\lambda)}{d\lambda} d\lambda,
\end{equation}
where $\Gamma(\epsilon)$ is the contour depicted in Figure~\ref{figure} and oriented counterclockwise, 
$$
	e(x,y) = -\frac{x+y}2\log \left(\frac{x+y}2\right)-\frac{x-y}2\log\left(\frac{x-y}2\right),
$$
and $D_L(\lambda)$ is the Toeplitz determinant of some symbol $\phi$ depending on the model. 
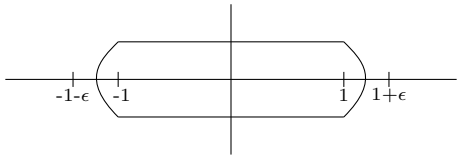
\begin{figure}
\begin{tikzpicture}
\draw (0,1) -- (0,-1);
\draw (-3,0) -- (3,0);
\draw (-1.5,-0.1) -- (-1.5,0.1); 
\draw (1.5,-0.1) -- (1.5,0.1);
\draw (2.1,-0.1) -- (2.1,0.1);
\draw (-2.1,-0.1) -- (-2.1,0.1);
\draw (-1.5,0.5) -- (1.5,0.5);
\draw (-1.5,-0.5) -- (1.5,-0.5);
\draw (-1.5,-0.5) .. controls (-2,0) and (-1.75,0.25) .. (-1.5,0.5);
\draw (1.5,-0.5) .. controls (2,0) and (1.75,0.25) .. (1.5,0.5);

\coordinate (A) at (-1.45,-0.2);
\coordinate (B) at (1.5,-0.2);
\coordinate (C) at (2.1,-0.2);
\coordinate (D) at (-2.1,-0.2);

\node[yshift=-0.1] at (A) {\tiny -1};
\node[yshift=-0.1] at (B) {\tiny 1};
\node[yshift=-0.1] at (C) {\tiny 1+$\epsilon$};
\node[yshift=-0.1] at (D) {\tiny -1-$\epsilon$};

\end{tikzpicture}
\caption{The contour $\Gamma(\epsilon)$ of the integral in \eqref{e:entropy-rep}.}\label{figure}
\end{figure}
In the XX model, $\phi$ is a scalar symbol and the standard theory of Toeplitz determinants apply. In the $XY$ model, the symbol is matrix-valued and given by
\begin{equation}\label{e:entropy-symb}
	\phi(\theta) = \begin{pmatrix}
	i\lambda & g(\theta)\\
	-g(\theta)^{-1} & i\lambda
	\end{pmatrix},
\end{equation}
where
$$
	g(\theta) = \frac{\alpha\cos \theta - 1 - i\gamma \sin\theta}{|\alpha \cos \theta - 1 - i \gamma \sin\theta|}.
$$
In \cite{IJK} and \cite{IMM}, the entropy of the XY model and its generalization, respectively, is computed using the Szeg\H{o}-Widom limit theorem \eqref{Sz.W} when $\phi$ in \eqref{e:entropy-symb} is sufficiently nice. However, in critical cases, such as when $\alpha=1$, the matrix-valued symbol $\phi$ has jumps and the Szeg\H{o}-Widom limit theorem no longer applies. This motivates the study of the asymptotics of Toeplitz determinants with piecewise continuous matrix-valued symbols, which we have initiated in this work and in particular we discuss the specific results in the next section. Further, as will be discussed in Section \ref{sec:Examples}, our results cover the critical case $\alpha=1$, providing the asymptotics of $D_L(\lambda)$ in \eqref{e:entropy-rep} when $\lambda\notin [-1,1]$, and pave the way for further study in this direction. It is also worth noting that when the chain is non-contiguous, as in \cite{GIKMV}, for example, it is no longer possible to deduce the study of the asymptotics of Toeplitz determinants directly and instead one needs to deal with certain block structures where, nevertheless, Toeplitz matrices with piecewise continuous matrix-valued symbols still appear but they are not in the scope of our present work.

In a related work of Ares et al.~\cite{AEFQ1} the authors consider the R\'enyi entanglement entropy for quadratic spinless fermionic chains with complex finite-range interactions, which leads to the asymptotic study of Toeplitz determinants with piecewise continuous matrix-valued symbols. More precisely, their work includes the study of quantum spin chain models with Dzyaloshinski-Moriya coupling and a Kitaev fermionic chain with long-range pairing. As in the previous works discussed above, a formula similar to \eqref{e:entropy-rep} is used to compute the entropy $S_\alpha(X)$ of the subsystem  $X$ with a particular choice of the Toeplitz determinant $D_X(\lambda)$, where $\alpha \in [0,1)$ and the limit $\alpha\to 1$ provides the von Neumann entropy discussed above. More precisely, in~\cite{AEFQ1}, it is argued that the entropy is given by
\begin{equation}\label{e:Renyi}
	S_\alpha(X) = \lim_{\epsilon \to 0} \frac{1}{4\pi i} \int_{\Gamma}
	f_\alpha(1+\epsilon, \lambda) \frac{d\log D_X(\lambda)}{d\lambda}d\lambda,
\end{equation}
where 
$$
	f_\alpha(x,y) = \frac{1}{1-\alpha} \log\left[ \left(\frac{x+y}2\right)^\alpha
	+\left(\frac{x-y}2\right)^\alpha\right]
$$
and $\Gamma$ is similar to the contour of integration in \eqref{e:entropy-rep}---we omit the full details and instead focus our attention on the block Toeplitz matrix $D_X(\lambda)$ that appear in \eqref{e:Renyi}. Indeed, in Section~\ref{sec:Examples}, we write down the matrix symbol in \eqref{e:Renyi-symb} and then proceed to analyze the corresponding asymptotics using our main results. It turns out that our findings are indeed in agreement with those obtained less rigorously in \cite{AEFQ1}.

\section{Basic definitions and statement of the main results}
\label{sec:main-results}

We denote by $(\ell^2)^N=\ell^{2}(\Z_+)^{N}$ the space of all $\C^N$-valued sequences $\{x_{n}\}_{n=0}^{\infty}$ 
equipped with the usual $2$-norm, which can be identified with direct sum of $N$ copies of $\ell^2(\Z_+)$.
Likewise, $L^\iy(\T)^{N\times N}$ stands for the space of all essentially bounded $\C^{N\times N}$-valued functions on $\T$, which can be identified with the space of all $N\times N$ matrices with entries from $L^\iy(\T)$. 

Given a bounded symbol $a\in L^\iy(\T)\NN$, the Toeplitz operator $T(a)$ and Hankel operator $H(a)$ are the bounded linear operators defined on $(\ell^{2})^N$ via the matrix representations
$$
T(a) = (a_{j-k}),\quad 0\leq j,k < \infty, 
$$
and
$$
H(a) = (a_{j+k+1}),\quad 0\leq j,k < \infty.
$$
Therein, 
$$
a_{k}=\frac{1}{2\pi}\int_{0}^{2\pi} a(e^{i\theta})e^{-ik\theta}\, d\theta,\qquad k\in\Z,
$$
are the (matrix) Fourier coefficients $a_k\in\C\NN$ of the function $a$.

Throughout this paper, let $\Ga=\{\tau_1,\dots,\tau_R\}\subset \T$ be set of $R$ distinct points
taken from the unit circle. We allow the case of $R=0$, i.e., $\Gamma=\emptyset$.

Let $PC(\T;\Ga)$ stand for the set of piecewise continuous functions
$\phi:\T\to\C$ which are continuous on $\T\setminus\Ga$. In other words, $\phi$ can have only jump discontinuities
at the finitely many points $\tau_1,\dots,\tau_R\in\T$.  For the one-sided limits at the jumps we will use the notation
$$
\phi(t\pm0) = \lim_{\theta\to +0}  \phi(te^{\pm i \theta}).
$$

Let $I\subset\R$ be a subset with the property that it does not contain two numbers whose difference is a nonzero integer.
In this paper, only the case of the open interval $I=(-1/2,1/2)$ is of interest to us. The consideration of the general setting here comes with no extra effort and might prove useful
in dealing with other cases of the determinant asymptotics elsewhere.

We call the function $\phi\in PC(\T;\Ga)\NN$  {\em $I$-regular} if
\begin{itemize}
\item[(a)]
$\phi(t)$ is invertible for all $t\in\T\setminus \Gamma$,\\[-2ex]
\item[(b)]
for each  $1\le k\le R$, both $\phi(\tau_k+0)$ and $\phi(\tau_k-0)$ are invertible matrices,\\[-2ex]
\item[(c)] 
for each $1\le k\le R$, one can choose the matrix logarithm
\begin{equation}\label{f.Lk}
L_k=\frac{1}{2\pi i}\log\left(\phi(\tau_k+0)^{-1}  \phi(\tau_k-0) \right)
\end{equation}
such that the real parts of all of its eigenvalues lie in $I$.
\end{itemize}
The above condition on $I$ guarantees that the $L_k$'s are uniquely determined.
Therefore, it is possible to define the {\em $I$-winding number} of an $I$-regular function $\phi$,
\begin{align}\label{f.wind}
\wind(\phi;I)=-\sum_{k=1}^R \tr(L_k)+\frac{1}{2\pi i}\sum_{k=1}^R \Big[\Delta \log \det \phi(t)\Big]_{t=\tau_{k}+0}^{\tau_{k+1}-0}.
\end{align}
Here $\tau_{R+1}=\tau_1$, and $\Delta(\dots)$ denotes the continuous increment of the (continuous) logarithm of the determinant on the arc $(\tau_{k},\tau_{k+1})$. 
Only for the sake of this definition we assume that $\tau_1,\dots,\tau_R$ appear in this order on the 
unit circle, i.e.,  $\tau_1=e^{i\theta_1},\dots,\tau_R=e^{i\theta_R}$ with $0\le \theta_1<\dots<\theta_R<2\pi$.

For a continuous non-vaninishing scalar function $c\in C(\T)$, the (usual) winding number is defined by
\begin{equation}\label{wind.scalar}
\wind(c)=\frac{1}{2\pi i}  \Big[\Delta \log c(e^{i\theta})\Big]_{\theta=0}^{2\pi}.
\end{equation}
In the case $R=0$ (i.e., $\Gamma=\emptyset$) the definition \eqref{f.wind} comes down to 
$$
\wind(\phi;I)=\wind(\det\phi),
$$
i.e., the $I$-winding number of the (continuous and invertible) matrix function $\phi(t)$ equals the winding number of its determinant $\det \phi(t)$.

Basic properties regarding the notions of $I$-regularity and the $I$-winding number will be established in Section \ref{sec:I-reg}.

The following theorem, which will be proved in Section \ref{index proof},  establishes the equivalence of four conditions.
These conditions (with $\kappa=0$) will appear  as the regularity assumption in our main results.

\begin{theorem}\label{thm:index}
Let $I=(-1/2,1/2)$, $\kappa\in\Z$, and $\phi\in PC(\T;\Gamma)^{N\times N}$. Then the following conditions are equivalent:
\begin{enumerate}
\item[(i)]
$T(\phi)$ is Fredholm on  $(\ell^2)^N$ with index $\ind T(\phi)=-\kappa$.
\\
\item[(ii)]
$T(\tilde{\phi})$ is Fredholm on  $(\ell^2)^N$ with index $\ind T(\tilde{\phi})=\kappa$.
\\
\item[(iii)]
$\phi$ is $I$-regular and $\kappa=\wind(\phi;I)$.
\\
\item[(iv)]
$\phi$ is $I$-regular and $\kappa=\wind(c)$.
\end{enumerate}
Therein, $c$ is the continuous and nonvanishing function on $\T$ defined by
\begin{equation}\label{fct.c}
c(t)=\frac{\det \phi(t)}{\prod_{k=1}^R u_{\beta_k,\tau_k}(t)}
\end{equation}
with $\beta_k=\tr L_k$ and the $L_k$'s given by \eqref{f.Lk}, and the functions $u_{\beta, \tau}$ are defined in \eqref{u-beta}.
\end{theorem}

Let us remark that if $\phi$ is invertible in $L^\infty(\T)\NN$, then in condition (ii) the operator $T(\tilde{\phi})$ can be replaced by the operator $T(\phi^{-1})$ (see Proposition \ref{p.T-Fred}).

This theorem rephrases the well-known criteria for Fredholmness of block Toeplitz operators on $(\ell^2)^N$
with piecewise continuous symbols in terms of $I$-regularity in the case of finitely many jump discontinuities. 
More importantly, it provides an explicit way to determine the Fredholm index
either via the $I$-winding number in (iii) or via the winding number of a scalar function in (iv).

We are aware of two further, but different approaches to compute the Fredholm index of the block Toeplitz operator
 with piecewise continuous matrix symbol. One, which is somewhat similar, can be found in the monograph  by Gohberg, Goldberg, Kaashoek \cite[Sect.~XXV.3]{GGK}. Another one, which can be applied to 
 to a much larger class of symbols but is perhaps less explicit, is based on approximate identities. It can be found in the monograph by B\"ottcher and Silbermann \cite[Sect.~4.27-4.31]{BS06}
(see also the references and comments therein).

\medskip
To specify the smoothness condition in our main results we introduce two classes of functions, which generalize the familiar class $C^{1+\eps}(\T)$  of differentiable functions with a H\"older-Lipschitz continuous derivative of order $0<\eps<1$.

\begin{definition}\label{PC1+e_def}
Let $PC^{1+\eps}(\T;\Gamma)$ stand for the set of all functions $a\in PC(\T;\Gamma)$ for which $a$ is continuously differentiable on 
$\T\setminus\Gamma$ and has a derivative satisfying a H\"older-Lipschitz condition of order $\eps>0$ on each arc 
$(\tau_k,\tau_{k+1})$, $1\le k\le R$ with $\tau_{R+1}=\tau_1$.
Here, as before,  $\tau_1,\dots,\tau_R$ appear in this order on the 
unit circle, i.e.,  $\tau_1=e^{i\theta_1},\dots,\tau_R=e^{i\theta_R}$ with $0\le \theta_1<\dots<\theta_R<2\pi$.
Furthermore, let 
\begin{equation}\label{f.PC1+e}
\CpwTG=PC^{1+\eps}(\T;\Gamma)\cap C(\T),
\end{equation}
which is the class of continuous functions with a piecewise H\"older-Lipschitz derivative.
\end{definition}

Both $PC^{1+\eps}(\T;\Gamma)$ and $\CpwTG$ are Banach algebras with the norm
$$
\|a\| =
\|a\|_\infty +\sum_{k=1}^R \sup\limits_{\theta_k<x<y<\theta_{k+1}} \frac{|a'(e^{ix})-a'(e^{iy})|}{|x-y|^{\eps}},
$$
where $\theta_{R+1}=\theta_1+2\pi$.

For a matrix $B\in\C\NN$ introduce the piecewise continuous matrix function  with a single jump discontinuity 
at $\tau\in\T$ by
\begin{equation}\label{u-B}
u_{B,\tau}(t)=(-t/\tau)^{B}=\exp(iB\arg(-t/\tau)),\qquad t\in\T.
\end{equation}
Here $|\!\arg(\cdot)|<\pi$.
The function $u_{B,\tau}$ is  the matrix  analogue of the scalar function $u_{\beta,\tau}$ defined in \eqref{u-beta}.

Our main results concerning the asymptotics of $\det T_n(\phi)$ for piecewise continuous matrix symbols
$\phi$ are as follows. Note that the description of the asymptotics requires a product representation of the symbol
$\phi$ which is the matrix analogue of \eqref{f.prod_sc}. The existence of this representation will therefore be part of the theorem.

\begin{theorem}\label{main-3}
Let $\phi\in PC^{1+\eps}(\T;\Gamma)\NN$. Assume that one (hence all) of the equivalent conditions (i)-(iv) in Theorem \ref{thm:index} hold with $\kappa=0$. 
Then $\phi$ admits a unique representation of the form
\begin{align}\label{f.prod0}
\phi(t)&=\phi_0(t)\phi_1(t)\cdots \phi_R(t)
\end{align}
where $\phi_0\in\CpwTG\NN$ is an invertible function with $\wind(\det \phi_0)=0$
and 
$$
\phi_k(t)=u_{B_k,\tau_k}(t),\qquad 1\le k\le R,
$$
with the matrices $B_k\in\C\NN$ having the property that the real parts of all their eigenvalues $\beta_k^{(1)},\dots,\beta_k^{(N)}$ are contained in the interval $I=(-1/2,1/2)$.

Moreover, 
\be\label{e:asymp}
\lim_{n\to\infty} \frac{\det T_n(\phi)}{G^n n^\Omega} =E
\ee
where 
\begin{align}\label{G.con1}
G =& \exp\left(\frac{1}{2\pi} \int_0^{2\pi} (\log \det \phi_0)(e^{ix})\, dx\right),
\\  \label{Om.con1}
\Omega =& - \sum_{k=1}^R\sum_{j=1}^N \, (\beta^{(j)}_k)^2,
\\
E =& \prod_{k=1}^R\prod_{j=1}^N G(1+\beta^{(j)}_k) \,G(1-\beta^{(j)}_k) \nn
\\
&\times \det\Big( T(\phi)T(\phi_R)^{-1}\cdots T(\phi_1)^{-1} T(\phi_1^{-1})^{-1}\cdots T(\phi_R^{-1})^{-1} T(\phi^{-1})\Big).
 \label{f.constE1}
\end{align}
\end{theorem}

As will be seen below (see Proposition \ref{p.prod} and formula \eqref{LkBkSim}) the matrices $B_k$ are similar to the matrices $L_k$. However, due to non-commutativity in the block case, they are in general not equal to each other except for the last ones, $B_R=L_R$.

Note that the first part of the constant $E$ features the Barnes $G$-function, an entire function defined by
\begin{equation}\label{Barnes}
G(1+z)=(2\pi)^{z/2}e^{-(z+1)z/2-\gamma_E z^2/2}\prod_{k=1}^\infty \left(\left(1+\frac{z}{k}\right)^ke^{-z+z^2/(2k)}\right)
\end{equation}
with $\gamma_E$ being Euler's constant. Note that this part of the constant $E$ is always nonzero under our assumptions.

The second part of the constant $E$ is a well-defined operator determinant,
i.e., it is the determinant of an operator of the form identity plus a trace class operator. 
In particular, the Toeplitz operators $T(\phi_k)$ and $T(\phi_k^{-1})$, $1\le k\le R$,  appearing therein are invertible. Note that 
$T(\phi_0)$ and $T(\phi_0^{-1})$ do not occur in the product. In fact, it need not be the case that $T(\phi_0)$, $T(\phi_0^{-1})$, $T(\phi)$, or $T(\phi^{-1})$ are invertible.  Our assumptions only imply that these four operators are Fredholm operators with index zero.  
To see this we can refer to Theorem \ref{thm:index} and Proposition \ref{p.T-Fred} below.
It is therefore possible that the operator-determinant (and hence the constant $E$) is zero, namely
when $T(\phi)$ or $T(\phi^{-1})$ is not invertible.

In the case of no jump discontinuities (i.e.,  $R=0$  and $\Gamma=\emptyset$) the previous theorem comes of course down to the Szeg\H{o}-Widom limit theorem.
Already in this case, no other general explicit expression is known for the operator determinant in the constant $E$ in the block case ($N\ge2$).
For certain very special classes the computation of $E$ can be done, such as for the smooth matrix-valued symbol discussed above in \eqref{e:entropy-symb} with $\alpha<1$ an expression was found using rather involved computations and Riemann-Hilbert analysis in \cite{IMM}. See also \cite[Sect.~10]{DIK13} for a review of some situations where effective evaluations have been obtained.

If one is not interested in the description of $E$, the formulation of the main result can be simplified.
One does not need the product representation \eqref{f.prod0} and the expressions for the constants $G$ and $\Omega$ can be stated differently.

\begin{corollary}\label{c.main-2}
Let $\phi\in PC^{1+\eps}(\T;\Gamma)\NN$. 
Assume that one (hence all) of the equivalent conditions (i)-(iv) in Theorem \ref{thm:index} hold with $\kappa=0$. 
Then the asymptotics \eqref{e:asymp} holds with the constants 
\begin{align}\label{G.con0}
G &= \exp\left(\frac{1}{2\pi} \int_0^{2\pi} (\log c)(e^{ix})\, dx\right),
\\ \label{Om.con0}
\Omega &= -\sum_{k=1}^R \tr\left( (L_k)^2\right),
\end{align}
where the $L_k$'s are given by \eqref{f.Lk} and the function $c$ is defined in 
\eqref{fct.c}.

Moreover, the constant $E$ is nonzero if and only if both operators
$T(\phi)$ and $T(\phi\iv)$ are invertible on $(\ell^2)^N$.
\end{corollary}

\begin{remark}\label{main-rem}
It is clearly desirable to know whether the constant $E$ vanishes or not, since only if it is nonzero the actual asymptotic behavior of $\det T_n(\phi)$ is
given by \eqref{e:asymp}. We will mention here two sufficient conditions for the invertibility of $T(\phi)$ on $(\ell^2)^N$.

If $\phi\in (L^\iy(\T))\NN$ is {\em sectorial}, then $T(\phi)$ is invertible on $(\ell^2)^N$. The function $\phi(t)$ being sectorial means that 
there exist invertible matrices $B,C\in\C\NN$ and some $\delta>0$ such that 
$$
\Re \langle B\phi(t)Cx,x\rangle \ge \delta \|x\|^2
$$
for all $x\in \C^N$ and for a.e.~$t\in\T$. Here the Euclidean inner product and norm in $\C^N$ are used. This condition is equivalent to the 
existence of (possibly different) invertible $B,C\in\C\NN$ and $\delta>0$ such that 
$$
\| I_N - B\phi(t)C\|_{\C\NN} \le 1-\delta
$$
for a.e.~$t\in\T$. For details on the notion of sectoriality and its generalizations we refer to \cite[Section~3.1]{BS06}.
Notice that for instance, strictly positive definite matrix functions $\phi$ are sectorial.

The other, somewhat peculiar sufficient condition we want to mention is the following.
If $\phi\in (L^\iy(\T))\NN$ satisfies the condition 
$$
(\phi(t))^* \phi(t^{-1})=I_N,\qquad \mbox{for a.e.~}t\in\T,
$$
then the kernel of both $T(\phi)$ and its adjoint $(T(\phi))^*$ on $(\ell^2)^N$ are trivial.
Here $(\phi(t))^*=(\overline{\phi(t)})^T$ is the complex adjoint function. This result is due to 
Voronin~\cite{V07} (see also~\cite{ES} for further details and generalizations). If, in addition, $T(\phi)$ is Fredholm, then we can conclude that $T(\phi)$ is invertible on $(\ell^2)^N$. 
Note that Fredholm criteria are known for piecewise continuous functions $\phi$.
\end{remark}

\section{Preliminaries and auxiliary results}\label{preliminaries}

\subsection{Properties of $\boldsymbol{I}$-regularity and $\boldsymbol{I}$-winding number}
\label{sec:I-reg}

Recall that the notions of $I$-regularity and the $I$-winding number have been defined in Section \ref{sec:main-results} for functions $\phi\in PC(\T;\Gamma)\NN$, see in particular \eqref{f.Lk} and \eqref{f.wind}. Here $I\subset\R$ is a subset with the property that it does not contain any two numbers whose difference is a nonzero integer.
In this paper, only $I=(-1/2,1/2)$ is of interest. The basic result about these notions are stated next.

\begin{proposition}\label{p1.1}
Let $\phi\in PC(\T;\Gamma)\NN$. Then
\begin{itemize}
\item[(i)]
$\wind(\phi;I)$ is a well-defined integer for any $I$-regular function $\phi$.
\end{itemize}
If $I$ is an open set, then 
\begin{itemize}
\item[(ii)]
$\wind(\phi;I)$ is invariant under continuous deformations of $I$-regular functions.
\end{itemize}
If $0\in I$, then 
\begin{itemize}
\item[(iii)]
$\wind(\phi_1\phi_2;I)=\wind(\phi_1;I)+\wind(\phi_2;I)$ provided $\phi_1$ and $\phi_2$ are $I$-regular functions having no discontinuities in common,
\\
\item[(iv)]
every invertible $\phi\in C(\T)^{N\times N}$ is $I$-regular and 
$$
\wind(\phi;I)=\wind(\det \phi),
$$
i.e., the $I$-winding number of $\phi$ coincides with the usual winding number \eqref{wind.scalar} of the scalar function $\det \phi$.
\end{itemize}
\end{proposition}
\begin{proof}
(i): It is straightforward to show that the exponential of $2\pi i$ times \eqref{f.wind} evaluates to one. 
Note that the matrix logarithms and the continuous increments are uniquely defined.

(ii): All quantities entering \eqref{f.wind}, in particular the matrix logarithms, depend continuously on $\phi$ in the $L^\infty$-norm if $I$ is open. Notice that this is no longer the case if we consider, e.g., the half-open interval
$[-1/2,1/2)$.

(iii): The second term in \eqref{f.wind} obviously splits additively if we apply it to the product $\phi_1\phi_2$.
In view of the first term, let $L_k$ be the matrix \eqref{f.Lk} for the product $\phi_1\phi_2$, and let
$L_k^{(1)}$ and $L_k^{(2)}$ be the corresponding matrices for $\phi_1$ and $\phi_2$. Assume that, say, $\phi_1$ is continuous at 
$\tau_k$. Then $L_k^{(1)}=0$ because $0\in I$ and the matrices $L_k$ and $L_k^{(2)}$ are similar to each other.
This implies that $\tr L_k=\tr L_k^{(1)}+\tr L_k^{(2)}$.

(iv): Note that the corresponding $L_k=0$ because $0\in I$.

Notice that statements (iii)-(iv) may no longer be true if one considers, e.g., an open interval $I$ not containing $0$.
\end{proof}

For $B\in\C\NN$ and $\tau\in\T$ we have introduced the functions $u_{B,\tau}$ in \eqref{u-B}
as a generalization of the scalar functions $u_{\beta,\tau}$ defined in \eqref{u-beta}. 
The functions are smooth on $\T\setminus\{\tau\}$ and have a possible jump at
$t=\tau$. In fact, they belong to $PC(\T,\{\tau\})\NN$ and the definition can be restated as
\begin{align*}
u_{B,1}(e^{ix}) &=
\exp((x-\pi)i B),\qquad 0<x<2\pi,
\end{align*}
and
\begin{align*}
u_{B,\tau}(t)=u_{B,1}(t/\tau),\qquad t \in\T.
\end{align*}
In particular the one-sided limits at the jump $t=\tau$ evaluate to
$$
u_{B,\tau}(\tau+0)=\exp(-\pi i B), \qquad u_{B,\tau}(\tau-0)=\exp(\pi i B).
$$

In the scalar case, representations of $\phi$ as a product \eqref{f.prod_sc}
play a role for the description of the asymptotics of $\det T_n(\phi)$. 
We will now generalize this product representation to the matrix case. 

\begin{proposition}\label{p.prod}
Let $I\subset\R$ be an open interval of length at most one and assume that $0\in I$.
Suppose $\phi\in PC(\T;\Ga)\NN$ is $I$-regular. 
Then $\phi$ admits a representation of the form
\begin{align}\label{f.prod1}
\phi(t)&=\phi_0(t)u_{B_1,\tau_1}(t)\cdots u_{B_R,\tau_R}(t)
\end{align}
where $\phi_0\in C(\T)\NN$ is an invertible function and 
the real parts of all the eigenvalues of $B_k$ lie in the interval $I$.
Moreover,
$$
\wind(\phi;I)=\wind(\det \phi_0).
$$
The matrices $B_1,\dots,B_R$ and the function $\phi_0$ are uniquely determined by $\phi$ and $I$.
\end{proposition}

Notice that due to the non-commutativity in the matrix case, the order of the factors in the product \eqref{f.prod1} matters. As already noted, the $B_k$'s are matrices similar to the $L_k$'s defined in \eqref{f.Lk}.

\begin{proof}
We prove the proposition by induction on the number $R$ of jump discontinuities $\tau_1,\dots,\tau_R$.
In case $R=0$ there is nothing to prove. We just take $\phi_0=\phi$ and observe Proposition \ref{p1.1}(iv).

 Now let $R\ge 1$ and assume that the statement has been proven for $R-1$.
Assume that $\phi$ has discontinuities at $\Ga=\{\tau_1,\dots,\tau_R\}$. We are going to show that we can write
$$\phi(t)=\psi(t) u_{B_R,\tau_R}(t)$$ where $\psi\in PC(\T;\Ga\setminus\{\tau_R\})$ is $I$-regular and $\wind(\psi;I)=\wind(\phi;I)$. This is all that is needed to apply the induction 
hypothesis to $\psi$ and finish the proof.

By assumption of $\phi$ being $I$-regular, we can find a matrix logarithm  
$$L_R=\frac{1}{2\pi i}\log\left( \phi(\tau_R+0)^{-1}\phi(\tau_R-0)\right)$$
with the real parts of all of its eigenvalues lying in $I$.
We put $B_R=L_R$ and observe 
$$
u_{B_R,\tau_R}(\tau_R+0)^{-1} u_{B_R,\tau_R}(\tau_R-0)=\exp(2\pi i B_R).
$$
Hence
$$
\phi(\tau_R+0)^{-1}\phi(\tau_R-0)=\exp(2\pi i L_R)= u_{B_R,\tau_R}(\tau_R+0)^{-1} u_{B_R,\tau_R}(\tau_R-0)
$$
and therefore
$$
\phi(\tau_R-0)u_{B_R,\tau_R}(\tau_R-0)^{-1}   = \phi(\tau_R+0) u_{B_R,\tau_R}(\tau_R+0)^{-1}.
$$
Introduce $\psi(t)=\phi(t) u_{B_R,\tau_R}(t)^{-1}$. By the preceding equality, this function is continuous at $t=\tau_R$.
In fact, $\psi\in PC(\T;\Ga\setminus\{\tau_R\})\NN$.
Evaluating the corresponding ``jump ratios'' for $\phi$ and $\psi$ at the points $\tau_1,\dots,\tau_{R-1}$ one notices that 
they are similar to each other, 
$$
\phi(\tau_k+0)^{-1}\phi(\tau_k-0)= T_k^{-1} \psi(\tau_k+0)^{-1}\psi(\tau_k-0) T_k,\qquad   T_k=u_{B_R,\tau_R}(\tau_k),
$$
which is due to the fact that $u_{B_R,\tau_R}(t)$ is continuous at $\tau_1,\dots,\tau_{R-1}$.
As a consequence the matrix logarithm of $\psi(\tau_k+0)^{-1}\psi(\tau_k-0)$ can be chosen to be similar to the matrix logarithm of
$\phi(\tau_k+0)^{-1}\phi(\tau_k-0)$, $1\le k<R$. This proves that $\psi$ is $I$-regular as well. 

Finally, we claim that $\wind(u_{B_R,\tau_R};I)=0$. Indeed, use  Proposition \ref{p1.1}(i)-(ii)
and employ a deformation argument to show that the map $\lambda\in[0,1]\mapsto \wind(u_{\lambda B_R,\tau_R};I)$ is constant. 
Now Proposition \ref{p1.1}(iii) implies $\wind(\psi;I)=\wind(\phi;I)$. Thus we have shown all that was needed
to apply the induction hypothesis.
\end{proof}

In principle, given $\phi$ and the corresponding matrix logarithms $L_k$, one can derive formulas expressing the $B_k$'s 
in terms of the $L_k$'s (and vice versa). These formulas show the similarity of these matrices explicitly. Indeed, starting with the product representation \eqref{f.prod1} and using the underlying definitions, we see that 
$$
e^{2\pi i L_k}=\phi(\tau_k+0)^{-1}\phi(\tau_k-0)=S_k^{-1} \phi_k(\tau_k+0)^{-1}\phi_k(\tau_k-0)S_k
=S_k^{-1}e^{2\pi i B_k} S_k,
$$
for $1\le k\le R$, thus
\be\label{LkBkSim}
L_k= S_k^{-1} B_k S_k
\ee
with
$$
S_k=\phi_{k+1}(\tau_k)\cdots\phi_{R}(\tau_k)=(-\tau_k/\tau_{k+1})^{B_{k+1}}\cdots(-\tau_k/\tau_{R})^{B_{R}}.
$$
While this shows that $B_R=L_R$, the relationship between all other terms becomes increasingly complicated and may be of little use practically.

It is also possible to prove the existence of product representations of similar kinds in which the order of the ``jump functions'' $u_{B_k,\tau_k}$ is permuted and/or where the continuous function $\phi_0$ occurs on the right  instead of on the left.
Due to non-commutativity, the $B_k$'s may have to be replaced by similar ones and $\phi_0$ may be different as well. For instance, under the same assumptions one can prove the existence of a product representation
$$
\phi(t)=u_{\wh{B}_R,\tau_R}(t)\cdots u_{\wh{B}_1,\tau_1}(t)\wh{\phi}_0(t)
$$
with $\wh{B}_k\cong L_k\cong B_k$. For the purpose of this paper, we could have worked with any such representation, but for sake of
definiteness we will focus on \eqref{f.prod1}.

\begin{corollary}\label{p.prod.cor}
Suppose $\phi\in PC^{1+\eps}(\T;\Gamma)^{N\times N}$ admits a product representation \eqref{f.prod1}.
Then the factor $\phi_0\in  \CpwTG^{N\times N}$.
\end{corollary}
\begin{proof}
The issue is only the smoothness of $\phi_0$. By assumption,  $\phi_0$ in the product representation \eqref{f.prod1} is continuous.
On the other hand, $\phi_0$ can be expressed as a product of $\phi$ and 
the functions $(u_{B_k,\tau_k})^{-1}=u_{-B_k,\tau_k}$. Since each of these factors is in $PC^{1+\eps}(\T;\Gamma)^{N\times N}$
and since $PC^{1+\eps}(\T;\Gamma)$ is an algebra, it follows that $\phi_0$ belongs to $PC^{1+\eps}(\T;\Gamma)^{N\times N}$ as well.
Now it remains to apply \eqref{f.PC1+e}.
\end{proof}

\subsection{Proof of Theorem \ref{thm:index}}\label{index proof}

We will now give the proof of Theorem \ref{thm:index}. For the issue of Fredholmness we rely on 
the known criteria as presented, for instance, in \cite[Sect.~XXV.3]{GGK}, Theorem 3.1, in particular.

Indeed, for $\phi\in PC(\T;\Gamma)\NN$ the Fredholmness of $T(\phi)$ on $(\ell^2)^N$ is equivalent to 
$$
\det \widehat{\Phi}(t,\mu)\neq 0
$$
for all $t\in\T$ and $\mu\in[0,1]$ where 
$$
\widehat{\Phi}(t,\mu) = \mu \phi(t+0)+(1-\mu)\phi(t-0)
$$
is an auxiliary function. For $t\in \T\setminus\Gamma$, the function $\widehat{\Phi}(t,\mu)=\phi(t)$ and the above 
condition amounts to the invertibility of $\phi(t)$ for all $t\in\T\setminus\Gamma$. 
On the other hand, for $t\in \Gamma$, we have $\widehat{\Phi}(t,0)=\phi(t+0)$ and
$\widehat{\Phi}(t,1)=\phi(t-0)$, which both have to be invertible. Furthermore, 
$\det \widehat{\Phi}(t,\mu)\neq0$ for all $\mu\in(0,1)$ if and only if 
$\det (\mu I_N + (1-\mu) S)\neq0$ for all $\mu\in(0,1)$ where $S=\phi(t+0)^{-1}\phi(t-0)$.
This means that none of the eigenvalues of $S$ can be a negative real number.
However, this is equivalent to saying that one can choose a matrix logarithm of $S$ such that all of the eigenvalues 
of $L=\frac{1}{2\pi i} \log S$ are contained in the interval $I=(-1/2,1/2)$. Combining all this shows that the 
Fredholmness of $T(\phi)$ is equivalent to $\phi$ being $I$-regular.

It is straightforward to check that in the case of $I=(-1/2,1/2)$, the $I$-regularity of $\phi$ is equivalent to the
$I$-regularity of $\tilde{\phi}$ where $\tilde{\phi}(t)=\phi(t^{-1})$.

Thus the equivalence of (i)-(iv) in Theorem \ref{thm:index} regarding Fredholmness is established.

Now we turn to the formula for the Fredholm index. One possiblity to prove this quickly is to use the product representation \eqref{f.prod1} along with a deformation argument. As shown in Proposition \ref{p.prod} such a
 product representation exist for every $I$-regular function $\phi$. We use it here with $I=(-1/2,1/2)$.

Given that product representation, consider a parameter $\lambda\in[0,1]$ and introduce the family of functions 
$$
\phi_\lambda(t) = \phi_0(t) u_{\lambda B_1,\tau_1}(t) \cdots u_{\lambda B_R,\tau_R}(t).
$$
Note that for $\lambda=0$ we obtain indeed the function $\phi_0\in C(\T)\NN$ appearing in the original product representation,
while for $\lambda=1$ we have $\phi_1=\phi$. The map $\lambda\in [0,1]\mapsto L^\infty(\T)\NN$
is continuous. Furthermore, due to the conditions on the $B_k$'s that the real parts of their eigenvalues lie in 
$I$, it follows that the same holds for the $\lambda B_k$'s. From this it follows that the functions $\phi_\lambda$ are $I$-regular. By what have shown above that means that all Toeplitz operators $T(\phi_\lambda)$ are Fredholm on
$(\ell^2)^N$. Therefore, when we deform $T(\phi_\lambda)$ along $\lambda\in [0,1]$, the Fredholm index remains
constant and this implies that 
$$
\ind T(\phi) =  \ind T(\phi_1) = \ind T(\phi_0).
$$
Furthermore, by Proposition \ref{p1.1}(ii) the $I$-winding number remains invariant under continuous
deformation of $I$-regular functions, which implies that 
$$
\wind(\phi;I)=\wind(\phi_1;I)=\wind(\phi_0;I) =\wind (\det \phi_0),
$$
where the latter is inferred from  Proposition \ref{p1.1}(iv). To complete the argument we remark that 
for every continuous and invertible function $\phi_0\in C(\T)\NN$ the Fredholm index is given by
$$
\ind T(\phi_0) = -\wind(\det \phi_0),
$$
the proof of which is not completely trivial (see \cite[Sect. XXIII.5]{GGK}). This proves that 
$$
\ind T(\phi) = -\wind(\phi;I).
$$
The equality of $\wind(\phi;I)= \wind(c)$ follows because the definition of $c$ implies that 
$c=\det \phi_0$. Finally, we can  conclude by analogy that $\ind T(\tilde{\phi}) = \wind(\phi;I)$  noting that 
$$
\ind T(\tilde{\phi})=-\wind(\tilde{\phi};I)=-\wind(\det \tilde{\phi}_0)=\wind(\det \phi_0)=\wind(\phi;I).
$$ 
In summary we have seen that all the expressions for $\kappa$ in (i)-(iv) of Theorem \ref{thm:index} coincide.
This concludes the proof.
 \hfill $\Box$

\subsection{Operator-theoretic preliminaries}

It is well-known and not difficult to prove that Toeplitz and Hankel operators satisfy the fundamental identities
\be\label{T1} T(ab) = T(a)T(b) + H(a)H(\tilde b)
\ee
and
\be\label{H1}
H(ab) = T(a)H(b) + H(a)T(\tilde b) .
\ee
In the last two identities and in what follows 
\be\label{tilde}
\tilde{b} (e^{i\theta}) = b(e^{-i\theta}) .
\ee
It is worthwhile to point out that these identities imply that 
\be\label{T2}
T(abc)=T(a)T(b)T(c),\qquad H(ab\tilde{c})=T(a)H(b)T(c)
\ee
for $a,b,c,\in L^\iy(\T)\NN$ if $a_n=c_{-n}=0$ for all $n>0$.

We define the projection $P_{n}$  by 
\[ 
P_{n} : \{x_k\}_{k=0}^\iy\in (\ell^2)^N \mapsto \{y_k\}_{k=0}^\iy\in (\ell^2)^N,\qquad
y_k=\left\{\ba{cc} x_k & \mbox{if } k<n\\ 0 & \mbox{if } k\ge n,\ea\right.
\]
and put $Q_n=I-P_n$. We remark that the image of $P_n$ can be identified with $(\C^n)^N\cong (\C^N)^n\cong \C^{nN}$.
Below we will identify operators of the form $P_nAP_n$ with $nN\times nN$ matrices. On the other hand, we will also think
of $nN\times nN$ matrices $A_n$ as linear operators on $(\ell^2)^N$.

In addition to the projections $P_n$ and $Q_n$ we need
\[ 
W_{n} : \{x_k\}_{k=0}^\iy\in (\ell^2)^N \mapsto \{y_k\}_{k=0}^\iy\in (\ell^2)^N,\qquad
y_k=\left\{\ba{cc} x_{n-1-k} & \mbox{if } k<n\\ 0 & \mbox{if } k\ge n.\ea\right.
\]
Note that $W_n^2=P_n$ and 
\begin{equation}\label{WTW}
W_n T_n(a) W_n= T_n(\ta).
\end{equation}

The following useful lemmas will be needed in what follows. 

\begin{lemma}[{\cite[Prop.~2.1]{Wi76}}]\label{l1.1b}
Let $B$ be a trace class operator and suppose that $A_n$ and $C_n$
are sequences such that $A_n\to A$ and $C_n^*\to C^*$ strongly. Then 
$A_nBC_n\to ABC$ in the trace class norm.
\end{lemma}

\begin{lemma}[{\cite[Lemma 9.3]{E03}}]\label{l1.KL}
Let $A_n = P_n+P_nKP_n+W_nLW_n+C_n$ be a sequence of $nN\times nN$
matrices where $K$ and $L$ are trace class operators, and $C_n$ tends
to zero in the trace class norm.  Then $\lim\limits_{n\to\infty}\det
A_n=\det(I+K)\det(I+L)$.
\end{lemma}

\begin{lemma}[{\cite[formula (1.4)]{Wi76}}]\label{widom}
 For bounded symbols $a,b\in L^\iy(\T)^{N\times N}$, 
$$T_{n}(ab) = T_{n}(a)T_{n}(b) + P_{n}H(a)H(\tilde{b} )P_{n} + W_{n}H(\tilde{a})H(b)W_{n}.$$
\end{lemma}

We say that a sequence of matrices $A_n\in \cL(\mathrm{Im} P_n)\subset (\ell^2)^N$ is {\em stable} if and only if there is an 
$n_0$ such that $A_n$ is invertible whenever $n\ge n_0$ and 
\begin{equation}\label{stability}
\sup_{n\ge n_0} \| A_n^{-1}\|_{ \cL(\mathrm{Im} P_n)}<+\infty.
\end{equation}
We will use stability in connection with the following basic facts. 

\begin{lemma}[{\cite[Section 6.2]{BS99}}]\label{stab.conv}
Let $a\in L^\iy(\T)^{N\times N}$ and assume that $T_n(a)$ is stable on $(\ell^2)^N$. Then
$T(a)$ is invertible on $(\ell^2)^N$ and 
$$
T_n(a)^{-1}\to T(a)^{-1},\qquad \left(T_n(a)^{-1}\right)^*\to \left(T(a)^{-1}\right)^*
$$
strongly on $(\ell^2)^N$.
\end{lemma}

\begin{proposition}[{see, e.g., \cite[Thm.~7.20]{BS06}, \cite[Thm.~6.9]{BS99},  or \cite{GF}}]\label{stab.cont}
Given $a \in C(\T)^{N\times N}$, the stability of $T_n(a)$ is equivalent to the invertibility of both
$T(a)$ and $T(\ta)$ on $(\ell^2)^N$.
\end{proposition}

The previous result can be generalized to $a\in PC\NN$ (see \cite[Cor.~6.12]{BS99}). For completeness' sake, we mention also
the following criterion, which is related to Theorem \ref{thm:index}.

\begin{proposition}[{see \cite[Thm.~6.5]{BS99} or \cite[Sect.~2.41-42, 2.94]{BS06}}] \label{Tiv.C}
Given $a\in C(\T)\NN$, the Toeplitz operator $T(a)$ is Fredholm on $(\ell^2)^N$ if and only if $\det a$ does not vanish on $\T$.
In this case, $\ind T(a) =- \wind(\det a)$.
\end{proposition}

The following result about Fredholmness is probably known, but we could not find a reference 
that pertains in particular to the index equality.

\begin{proposition}\label{p.T-Fred}
Let $a\in L^\infty(\T)\NN$ be invertible. Then $T(\tilde{a})$ is Fredholm on $(\ell^2)^N$ if and only if 
$T(a^{-1})$ is Fredholm on $(\ell^2)^N$. Moreover, in this case,
$$
\ind T(\ta) = \ind T(a\iv).
$$
\end{proposition}
\begin{proof}
The statements follow from the fact that the two operators $T(\ta)$ and $T(a\iv)$ are equivalent after extension.
By this it is meant  \cite{BaTs}  that there exist
Banach spaces $Z_1$ and $Z_2$
and invertible bounded linear operators $E_1$ and $E_2$ between the appropriate direct sum spaces such that 
$$
T(\ta)\oplus I_{Z_1} = E_1\Big( T(a\iv)\oplus I_{Z_2}\Big) E_2.
$$
Indeed, consider the following extensions of $T(\ta)$ and $T(a\iv)$ onto $(\ell^2(\Z))^N$, 
$$
T(\ta)\oplus I_{(\ell^2(\Z_-))^N}=P L(\ta) P + Q ,\quad 
T(a\iv)\oplus   I_{(\ell^2(\Z_-))^N}=P L(a\iv)P+Q.
$$
Here $L(b)\cong (b_{j-k})_{j,k=-\infty}^{\infty}$ stands for the Laurent operator on $(\ell^2(\Z))^N$, $P$ stands for the orthogonsl projection from  $(\ell^2(\Z))^N$ onto
$(\ell^2(\Z_+))^N=(\ell^2)^N$, $Q=I-P$ is the complementary projection, and $J:\{x_n\}_{n=-\infty}^{\infty} \mapsto  \{x_{-1-n}\}_{n=-\infty}^{\infty}$ is a flip operator on $(\ell^2(\Z))^N$. We have the relations $J^2=I$, $JQJ=P$, $JL(b)J=L(\tilde{b})$.
These extended operators can be written as
\begin{align*}
J\Big(P L(\ta) P + Q\Big)J 
&= QL(a)Q+P\\
&= \Big(L(a)Q+P\Big) \Big(I-PL(a)Q\Big),
\\
PL(a\iv)P+Q &=
 \Big(L(a\iv)P+Q\Big) \Big(I-QL(a\iv)P\Big)
 \\
 &=
 L(a\iv) \Big(L(a)Q+P\Big) \Big(I-QL(a\iv)P\Big),
\end{align*}
from which the equivalence is easily seen by noting that the operators $J$, $L(a\iv)$, as well as
$$
I-PL(a)Q\quad\mbox{ and }\quad I-QL(a\iv)P
$$
are invertible.
\end{proof}

\subsection{Invertibility, stability, and determinant asymptotics for pure jump symbols}

For pure matrix jump symbols $\phi=u_{B,\tau}$ we are going to state the invertibility of
$T(\phi)$ and $T(\wt{\phi})$ on $(\ell^2)^N$, and the stability of the sequence $T_n(\phi)$ under certain conditions on $B$.
In addition, we describe the asymptotics of the determinant $\det T_n(\phi)$ as $n\to\infty$.

As we will see, the matrix case completely reduces to the scalar case, for which these results are known. We refer to \cite[Thm.~5.62]{BS06} for  invertibility, to \cite[Cor.~2.19]{BS99} for stability, and to \cite[Cor.~10.60]{BS06} for the determinants. 
Actually, the results about the pure scalar symbols $u_{\beta,\tau}$ can also be seen directly. Indeed, if $|\Re\beta|<1/2$, then 
the symbol is sectorial, i.e., $\Re(u_{\beta,\tau}(t))\ge \varepsilon$ for some $\varepsilon>0$, which implies invertibility of $T(u_{\beta,\tau})$ and
$T(\tilde{u}_{\beta,\tau})$
and stability of $T_n(u_{\beta,\tau})$. The matrix $T_n(u_{\beta,\tau})$ is basically a Cauchy matrix and its determinant can be evaluated explicitly for any $\beta\in\C$.

\begin{proposition}\label{p.pure}
Assume that the eigenvalues $\beta^{(1)},\dots,\beta^{(N)}$ of an $N\times N$ matrix $B$ have real parts in $I=(-1/2,1/2)$. 
Then the operators $T(u_{B,\tau})$ and $T(\tilde{u}_{B,\tau})$ are invertible on $(\ell^2)^N$ and the sequence $T_n(u_{B,\tau})$ is stable.
Furthermore,
$$
\det T_n(u_{B,\tau})= E n^{\Omega} (1+o(1)), \quad \mbox{ as }n\to\infty
$$
with 
$$
\Omega =-\sum_{k=1}^N (\beta^{(k)})^2,\qquad  E=\prod_{k=1}^N G(1+\beta^{(k)})G(1-\beta^{(k)}),
$$
where $G(z)$ stands for the Barnes $G$-function \eqref{Barnes}.
\end{proposition}

\begin{proof}
First notice that $T(\tilde{u}_{B,\tau})=T(u_{-B,\bar{\tau}})$. Hence the invertibility for $T(\tilde{u}_{B,\tau})$ will follow once it is proven for 
$T(u_{B,\tau})$.

If $B=SJS^{-1}$ where $S$ is an invertible matrix and $J$ is another matrix (such as the Jordan normal form of 
$B$), then $u_{B,\tau}(t)=Su_{J,\tau}(t)S^{-1}$ and 
$$
T(u_{B,\tau})=(S\otimes I) T(u_{J,\tau}) (S^{-1}\otimes I), \qquad 
T_n(u_{B,\tau})=(S\otimes I_n) T_n(u_{J,\tau}) (S^{-1}\otimes I_n).
$$
Here $S\otimes I$ and $S\otimes I_n$ stand for the linear operators on $(\ell^2)^N\cong \C^N\otimes  \ell^2$
and $(\C^n)^N\cong  \C^N\otimes \C^n$ defined by
$$
S\otimes I:(x_0,x_1,x_2,\dots) \mapsto (Sx_0,Sx_1,Sx_2,\dots)
$$
and
$$
S\otimes I_n:(x_0,x_1,x_2,\dots x_{n-1}) \mapsto (Sx_0,Sx_1,Sx_2,\dots, Sx_{n-1})
$$
where $x_k\in\C^N$. On a more formal level, $S\otimes I$ and $S\otimes I_n$ are the block Toeplitz operator and  the $nN\times nN$ block Toeplitz matrix, resp., with
symbol equal to $S$, a constant $N\times N$ matrix function.

Hence, all issues can  be reduced to the case where $B$ is of Jordan normal form.
Notice first that things become particularly simple if $B$ is of diagonal form, say,
$$
B=\diag(\beta^{(1)},\dots, \beta^{(N)}).
$$
Then 
$$
u_{B,\tau}= \diag( u_{\beta^{(1)},\tau},\dots, u_{\beta^{(N)},\tau}),
$$
and a corresponding ``diagonal representation'' holds for the block Toeplitz operators and matrices. By assumption the real parts of all $\beta^{(k)}$'s are in 
$(-1/2,1/2)$, and therefore the known scalar results mentioned above imply the assertions.

Using a similar decomposition, it is easy to see that the general case where $B$ is of Jordan form can be reduced to the case where 
 $B$ is a simple Jordan block, say,
 $$
 B=\left(\begin{array}{cccc} \beta & 1 & \dots & 0  \\
 &\beta & \ddots  & \vdots \\
 && \ddots & 1 \\
 &&& \beta \end{array}\right)
 $$
with the real part of $\beta$ in $I=(-1/2,1/2)$. For  the function $u_{B,1}$, which is defined via a matrix exponential, we get
$$
u_{B,1}(e^{ix})=\exp(i(x-\pi)B)=u_{\beta,1}(e^{ix})\exp(i(x-\pi)J)
$$
(where $J$ is the  simple Jordan block with eigenvalue zero) and
$$
u_{B,\tau}(e^{ix}) = \left(\begin{array}{cccc} u_{\beta,\tau}& \ast & \dots & \ast  \\
 &u_{\beta,\tau} & \ddots  & \vdots \\
 && \ddots & \ast \\
 &&& u_{\beta,\tau}  \end{array}\right)
 $$
where ``$\ast$'' stands for certain piecewise continuous functions. A similar upper-triangular block matrix representation is obtained
for  $T(u_{B,\tau})$ and $T_n(u_{B,\tau})$. From there it is seen that for the issues of invertibility, stability, and for the determinants, only the entries on the diagonals matter. Therefore, again, everything reduces to the scalar case.
\end{proof}

\section{Determinant asymptotics: First results}\label{1st results}

\subsection{Basic localization results}

In \cite{B1979}, the following result about the product of Hankel operators was established.
{\em  Suppose $a$ and $b$ are bounded functions on $\T$ for which there exists a smooth partition of unity, $f+g=1$, such that both $af$ and $bg$  have derivatives satisfying a Lipschitz condition with order greater than $1/2$. Then 
 $H(a)H(\tilde{b})$ and $H(\tilde{a})H(b)$ are trace class. 
}
 
Using the same idea but specializing to piecewise continuous symbols we can somewhat improve on the exponent in the H\"older-Lipschitz condition.

The following result on the decay of the Fourier coefficients of functions in the two classes $PC^{1+\eps}(\T;\Gamma)$ and $\CpwTG$ (see Definition~\ref{PC1+e_def}) can be established easily. We will assume throughout what follows that $0<\eps<1$.

\begin{lemma}\label{l.Four-asym}
The Fourier coefficients of $f\in PC^{1+\eps}(\T;\Gamma)$ have the asymptotics
$$
f_n= O(|n|^{-1}),\qquad |n|\to+\infty,
$$
while those of $f\in \CpwTG$ have the asymptotics
$$
f_n= O(|n|^{-1-\eps}),\qquad |n|\to+\infty.
$$
\end{lemma}
 
\begin{lemma}\label{l.traceHH}
Suppose $a,b\in PC^{1+\eps}(\T;\Gamma)$ with $\eps>0$ such that $a$ and $b$ do not have discontinuities in common.  
Then  $H(a)H(\tilde{b})$ and $H(\tilde{a})H(b)$ are trace class. 
\end{lemma}

\begin{proof}   
By assumption $a$ is continuous on $\T\setminus \Gamma_a$ and $b$ is continuous on $\T\setminus \Gamma_b$ where
$\Gamma_a\cup\Gamma_b\subseteq \Gamma$ and $\Gamma_a$ and $\Gamma_b$ are disjoint. Hence there exists
a partition of unity, $f+g=1$, with both $f,g$ sufficiently smooth such that $af$ and $bg$ are continuous and thus  belong to
$\CpwTG$ (see also \eqref{f.PC1+e}). Consider the first Hankel product (the other one can be dealt with in the same way):
\begin{align*}
H(a)H(\tilde{b}) &= 
H(a)T(\tilde{f})H(\tilde{b}) + H(a)T(\tilde{g})H(\tilde{b}) 
\\
&=
\Big(T(a)H(\tilde{f})-H(af)\Big)H(\tilde{b})+H(a)\Big(H(\tilde{g})T(b)-H(\tilde{g}\tilde{b}) \Big).
\end{align*}
Here we used the identity \eqref{H1}.
Each of the operators appearing therein is bounded and the Hankel operators   $H(\tilde{f})$ and $H(\tilde{g})$ are trace class.
Therefore it suffices to show that the products $H(af)H(\tilde{b})$ and $H(a)H(\tilde{g}\tilde{b})$ are trace class.
Consider the first product (again the second one can be dealt with analogously). We write
$$
H(af)H(\tilde{b})= \Big(H(af)D_\eps\Big)\Big(D_{-\eps} H(\tilde{b})\Big)
$$
where $D_\eps=\diag_{j\ge0}((1+j)^{\eps/2})$ is a diagonal operator on $\ell^2(\Z_+)$.
We claim that both factors are Hilbert-Schmidt, hence their product is trace class, as desired.
Indeed, the Hilbert-Schmidt norms can be estimated as follows:
\begin{align*}
\|H(af)D_\eps\|_2^2  &\le C\sum_{j,k\ge0}(1+j+k)^{-2-2\eps}(1+k)^{\eps} <+\infty,
\\
\|D_{-\eps} H(\tilde{b})\|_2^2 &\le C\sum_{j,k\ge0}(1+j)^{-\eps} (1+j+k)^{-2} <+\infty.
\end{align*}
Therein we used the estimates on the Fourier coefficients for $af\in \CpwTG$ and $b\in PC^{1+\eps}(\T;\Gamma)$
stated in Lemma \ref{l.Four-asym}.
\end{proof}

In anticipation of using the previous trace class condition, we state first a general result.
We may think of the functions $\phi_0, \dots,\phi_R$ as having discontinuities at different locations and being sufficiently 
smooth away from the discontinuities.

\begin{proposition}\label{p3.1b}
Assume $\phi=\phi_0\phi_1\cdots \phi_R$ such that $H(a)H(\tilde{b})$ and $H(\tilde{a})H(b)$ are trace class whenever $a = \phi_{0}\cdots \phi_{k-1}$ and $b = \phi_{k}$, $1\le k\le R$. Then
\begin{align*}
K_1 &=T(\phi)-T(\phi_0)T(\phi_1)\cdots T(\phi_R)
\\
K_2 &=T(\tilde{\phi})-T(\tilde{\phi}_0)T(\tilde{\phi}_1)\cdots T(\tilde{\phi}_R)
\end{align*}
are trace class and 
\begin{align*}
T_n(\phi)=T_n(\phi_0)T_n(\phi_1)\cdots T_n(\phi_R) + P_nK_1P_n+W_nK_2W_n +
C_n
\end{align*}
where $C_n$ tends to zero in the trace norm.
\end{proposition}

\begin{proof}
For $R=0$ there is nothing to prove. The case $R=1$ is settled by the Widom's formula (Lemma \ref{widom}),
$$
T_n(ab)=T_n(a)T_n(b) + P_n H(a) H(\tilde{b})P_n + W_n H(\tilde{a})H(\tilde{b}) W_n.
$$
In view of \eqref{T1} notice that 
$$
H(a) H(\tilde{b})=T(ab)-T(a)T(b),\qquad
H(\tilde{a})H(b)=T(\tilde{a}\tilde{b})-T(\tilde{a})T(\tilde{b}),
$$
which proves the trace class property of $K_1$ and $K_2$ by invoking the assumption with $a=\phi_0$ and $b=\phi_1$.

By way of induction, assume that we have established
$$
T_n(\phi_0\cdots \phi_{R-1})=T_n(\phi_0)\cdots T_n(\phi_{R-1})+P_nK_1'P_n+W_n K_2'W_n + C_n'
$$
with $K_1', K_2'$ being trace class. Applying Widom's formula with $a=\phi_0\cdots \phi_{R-1}$ and $b=\phi_R$
it follows that with $K_{1}'' = H(\phi_0\cdots \phi_{R-1})H(\wt{\phi}_R), K_{2}'' = H(\widetilde{\phi_0\cdots \phi_{R-1}})H(\phi_R),$ 
\begin{align*}
T_n(\phi)
&= T_n(\phi_0\cdots \phi_{R-1})T_n(\phi_R)+P_n K_1''P_n +W_nK_2''W_n
\\
&= 
\Big(T_n(\phi_0)\cdots T_n(\phi_{R-1})+P_nK_1'P_n+W_n K_2'W_n + C_n'\Big) T_n(\phi_R)
\\
&\qquad
+P_n K_1''P_n +W_nK_2''W_n
\\
&=
T_n(\phi_0)\cdots T_n(\phi_{R})+C_n' T_n(\phi_R) +P_n K_1''P_n +W_nK_2''W_n
\\
&\qquad
+P_nK_1'T(\phi_R)P_n-P_nK_1'Q_nT(\phi_R)P_n
\\
&\qquad
+W_nK_2'T(\tilde{\phi}_R)W_n-W_nK_2'Q_nT(\tilde{\phi}_R)W_n.
\end{align*}
In the last two lines we used $P_n=I-Q_n$ and also \eqref{WTW}.
The terms containing $Q_n$ and $C_n'$ converge to zero in the trace norm by Lemma \ref{l1.1b}, so they will make up the term 
$C_n$. The operators $K_{1}''$ and $K_{2}''$ are trace class by assumption.  Consequently we obtain
$$
T_n(\phi)=T_n(\phi_0)T_n(\phi_1)\cdots T_n(\phi_R) + P_nK_1P_n+W_nK_2W_n +
C_n
$$
with 
\begin{align*}
K_1&=K_1''+K_1'T(\phi_R)
\\
&= T(\phi)-T(\phi_0\cdots\phi_{R-1})T(\phi_R)
\\
&\quad +\Big(T(\phi_0\cdots\phi_{R-1})-T(\phi_0)\cdots T(\phi_{R-1})\Big)T(\phi_R)
\\
&= T(\phi)-T(\phi_0)\cdots T(\phi_R)
\end{align*}
as desired.  The trace class property for $K_1$ follows since
$K_1'$ and $K_1''$ are trace class. An analogous argument yields the results for $K_2$.
\end{proof}

Applying Lemma \ref{l.traceHH} to the previous proposition we obtain the following result, which is our first step
towards the Toeplitz determinants with a symbol $\phi$ given by a representation \eqref{f.prod1}.

\begin{proposition}\label{p3.1}
Assume $\phi=\phi_0\phi_1\cdots \phi_R$ where $\phi_k=u_{B_k,\tau_k}$, $1\le k\le R$,  
and $\phi_0\in\CpwTG\NN$ with $\eps>0$. Then
the operators
\begin{align*}
K_1 &=T(\phi)-T(\phi_0)T(\phi_1)\cdots T(\phi_R)
\\
K_2 &=T(\tilde{\phi})-T(\tilde{\phi}_0)T(\tilde{\phi}_1)\cdots T(\tilde{\phi}_R)
\end{align*}
are trace class and 
\begin{align*}
T_n(\phi)=T_n(\phi_0)T_n(\phi_1)\cdots T_n(\phi_R) + P_nK_1P_n+W_nK_2W_n +
C_n
\end{align*}
where $C_n$ tends to zero in the trace norm.
\end{proposition}

We proceed with the general case and obtain, under certain assumptions, a localization result for 
Toeplitz determinants.

\begin{theorem}\label{t:asym:1}
Assume $\phi=\phi_0\phi_1\cdots \phi_R$ such that $H(a)H(\tilde{b})$ and $H(\tilde{a})H(b)$ are trace class whenever $a = \phi_{0}\cdots \phi_{k-1}$ and $b = \phi_{k}$, $1\le k\le R$. Suppose in addition that the sequence $T_n(\phi_k)$ is  stable for each $0\le k\le R$. Then 
 \begin{align}\label{f.asym2}
\lim_{n\to\infty} 
\frac{\det T_n(\phi)}{\prod_{k=0}^R \det T_n(\phi_k)}
=E
\end{align}
where $E=E_1E_2$ and 
\begin{align*}
E_1 &= \det T(\phi) T(\phi_R)^{-1}\cdots T(\phi_0)^{-1},
\\
E_2 &= \det T(\tilde{\phi}) T(\tilde{\phi}_R)^{-1}\cdots T(\tilde{\phi}_0)^{-1},
\end{align*}
and the operator determinants are well-defined.
\end{theorem}
\begin{proof} 
Using Lemma \ref{stab.conv} and noting that both $T_n(\phi)$ and $T_n(\tilde{\phi})=W_n T_n(\phi) W_n$ are stable (see \eqref{WTW})
 it follows that we have strong convergence on $(\ell^2)^N$ of
 $$
T_n(\phi_k)^{-1}\to T(\phi_k)^{-1},\qquad
T_n(\wt{\phi}_k)^{-1}\to T(\wt{\phi}_k)^{-1},
$$
and of the corresponding adjoints. The invertibility of $T(\phi_k)$ and $T(\wt{\phi}_k)$ is guaranteed as well.

Noting that the inverses exist for sufficiently large $n$, we can consider the sequence
$$
A_n = T_n(\phi) T_n(\phi_R)^{-1}\cdots T_n(\phi_0)^{-1}.
$$
From Proposition \ref{p3.1b} (and again  \eqref{WTW}) we can write this as
\begin{align*}
A_n &= 
P_n + P_n K_1P_n T_n(\phi_R)^{-1}\cdots T_n(\phi_0)^{-1}P_n
\\
&\qquad \quad+
W_n K_2 P_n T_n(\wt{\phi}_R)^{-1}\cdots T_n(\wt{\phi}_0)^{-1}W_n + C_n'
\end{align*}
with a certain $C_n'\to 0$ in trace norm and the trace class operators $K_1$ and $K_2$ taken from Proposition \ref{p3.1b}.
Using the strong convergence of the inverses (and their adjoints), it follows from  Lemma \ref{l1.1b} that
\begin{align*}
A_n &=P_n+P_nL_1P_n+W_n L_2W_n +C_n
\end{align*}
with  certain $C_n\to 0$ in trace norm and 
\begin{align*}
L_1 &= K_1 T(\phi_R)^{-1}\cdots T(\phi_0)^{-1}= T(\phi) T(\phi_R)^{-1}\cdots T(\phi_0)^{-1}-I,
\\
L_2 &= K_2 T(\tilde{\phi}_R)^{-1}\cdots T(\tilde{\phi}_0)^{-1}=T(\tilde{\phi}) T(\tilde{\phi}_R)^{-1}\cdots T(\tilde{\phi}_0)^{-1}-I,
\end{align*}
which are both trace class operators. Taking the determinant of $A_n$ and then passing to the limit gives the left hand side of \eqref{f.asym2}.
Invoking Lemma \ref{l1.KL} it follows that the limit equals the product of two well-defined operator determinants,
$$\lim\limits_{n\to\infty}\det A_n=\det(I+L_1)\det(I+L_2).$$
Using the previous expressions for $L_1$ and $L_2$ we arrive at \eqref{f.asym2}.
\end{proof}

\begin{theorem}\label{t:asym:1b}
Assume $\phi=\phi_0\phi_1\cdots \phi_R$ where $\phi_k=u_{B_k,\tau_k}$, $1\le k\le R$, and such that the eigenvalues of $B_{k}\in\C^{N\times N}$ have real parts in the interval $I = (-1/2, 1/2).$ Moreover, suppose also that $\phi_0\in \CpwTG^{N\times N}$ is such that both $T(\phi_0)$ and $T(\tilde{\phi}_0)$ are invertible on $(\ell^2)^N$.
Then 
\begin{align}\label{f.asym1}
\lim_{n\to\infty} 
\frac{\det T_n(\phi)}{G[\phi_0]^n \prod_{k=1}^R \det T_n(\phi_k)}
=E
\end{align}
where $E=E_1E_2E_3$ and 
\begin{align*}
E_1 &= \det T(\phi) T(\phi_R)^{-1}\cdots T(\phi_0)^{-1},
\\
E_2 &= \det T(\tilde{\phi}) T(\tilde{\phi}_R)^{-1}\cdots T(\tilde{\phi}_0)^{-1},
\\
E_3 &= \det T(\phi_{0})T(\phi_{0}^{-1}),
\end{align*}
and $G[\phi_0]$ is defined in \eqref{const.G}.
\end{theorem}
\begin{proof}
Applying Lemma \ref{l.traceHH} and noting that $\phi_k \in PC^{1+\eps}(\T;\Gamma)\NN$, we conclude that the 
assumption about the product of the Hankel operators in Theorem \ref{t:asym:1} are fulfilled. Moreover, the assumption on 
$B_k$ implies that the sequence $T_n(\phi_k)$ is stable by Proposition \ref{p.pure}. On the other hand,
the stability of the sequence $T_n(\phi_0)$ is due to Proposition \ref{stab.cont}.

Applying the Szeg\H{o}-Widom limit theorem (Theorem \ref{Sz.W}) gives the asymptotics 
$\det T_n(\phi_0)\sim G[\phi_0]^n E_3$. Notice that $\CpwTG\subset F$ and therefore the smoothness conditions are fulfilled.
\end{proof}

The asymptotics of the product term in the denominator of \eqref{f.asym1} can be stated explicitly using Proposition \ref{p.pure}.
Indeed, for $1\le k\le R$, let $\beta^{(j)}_k$ ($1\le j\le N$) be the eigenvalues of $B_k$, multiplicities taken into account.
Then
\begin{align}\label{f4.3}
\prod_{k=1}^R \det T_n(\phi_k)
&\sim
n^\Omega E_4
\end{align}
with 
\begin{align}\label{f4.4}
\Omega &= - \sum_{k=1}^R\sum_{j=1}^N (\beta^{(j)}_k)^2
\\ \label{f4.5}
E_4 &= \prod_{k=1}^R\prod_{j=1}^N G(1+\beta^{(j)}_k)G(1-\beta^{(j)}_k).
\end{align}

In  the asymptotic formulas below, we will keep the determinant product on the left hand side of \eqref{f4.3}
for sake of simplicity.

\subsection{An operator determinant identity}

Our next goal is to express the operator determinant $E_2$ in a different way in order to be able to combine it with the expressions for $E_1$ and $E_3$ and write $E$ as a single operator determinant.

As a lemma we need a relationship between the inverses of $T(\ta)$ and $T(a^{-1})$, 
which is in some sense already suggested by Proposition \ref{p.T-Fred} and its proof.

\begin{lemma}\label{lem.TinvH}
Let $a\in L^\infty(\T)\NN$ be invertible. Then $T(\ta)$ is invertible on $(\ell^2)^N$ if and only if
$T(a^{-1})$ is invertible on $(\ell^2)^N$. Moreover,
\begin{align}\label{f.TivH}
0 &=T(\tilde{a})^{-1}H(\tilde{a})+ H(\tilde{a}^{-1})T(a^{-1})^{-1}
\end{align}
and 
\begin{align}\label{f.Tiv2}
T(a^{-1})^{-1} &=T(a)-H(a)T(\ta)^{-1}H(\ta).
\end{align}
\end{lemma}
\begin{proof}
By \eqref{H1} we have the identity
$$
0=T(\ta)H(\ta^{-1})+H(\ta) T(a^{-1}).
$$
Assuming invertibility of both $T(\ta)$ and $T(a^{-1})$, the first identity \eqref{f.TivH} follows. Now assume 
that only $T(\ta)$ is invertible. Then, using \eqref{T1} and what we just stated,
\begin{align*}
\Big(T(a)-H(a)T(\ta)^{-1}H(\ta)\Big)T(a^{-1})
&=
T(a)T(a^{-1})+H(a)T(\ta)^{-1}T(\ta)H(\ta^{-1})
\\
&=
T(a)T(a^{-1})+H(a)H(\ta^{-1})=I.
\end{align*}
In a similar manner, multiplication from the other side gives the identity. We conclude that if $T(\tilde{a})$ is invertible, then 
$T(a^{-1})$ is invertible and its inverse is given by \eqref{f.Tiv2}.
Finally, we can make an analogous argument that the invertibility of $T(a^{-1})$ implies the invertibility of $T(\ta)$.
\end{proof}

\begin{proposition}\label{id:det-2}
Let $a,b\in L^\iy(\T)^{N\times N}$ be such that the operators $T(\ta)$ and $T(\tb)$ are invertible on $(\ell^2)^N$. Assume moreover that
$H(\ta)H(b)$ and $H(b^{-1})H(\ta^{-1})$ are trace class. Then the following two operator determinants are well-defined and coincide:
\begin{equation}\label{det-det-2}
\det T(\ta\tb) T(\tb)\iv T(\ta)\iv = \det  T(a\iv)\iv T(b\iv)\iv T(b\iv a\iv).
\end{equation}
\end{proposition}
\begin{proof}
Note that the invertibility of the operators $T(\ta)$ and $T(\tb)$ implies the invertibility of the symbols $\ta$ and $\tb$. Hence
$a^{-1},b^{-1}\in L^\iy(\T)\NN$ as well. 
Lemma \ref{lem.TinvH} now implies that the operators $T(a\iv)$ and $T(b\iv)$ are invertible.
Applying \eqref{T1} yields
\begin{align*}
T(\ta\tb) &=T(\ta)T(\tb)+H(\ta)H(b),\\
T(b\iv a\iv) &=T(b\iv)T(a\iv)+H(b\iv)H(\ta\iv),
\end{align*}
and multiplying with the appropriate inverses we conclude that the operator products in \eqref{det-det-2} are of the form
identity plus trace class. Hence both determinants are well-defined. As we can move invertible operators from one side to the other we
see that 
\begin{align*}
\det T(\ta\tb) T(\tb)\iv  T(\ta)\iv &=\det \Big(I+T(\ta)\iv H(\ta)H(b)T(\tb)\iv\Big),\\
\det T(a\iv)\iv T(b\iv)\iv T(b\iv a\iv) &=\det \Big(I+T(b\iv)\iv H(b\iv)H(\ta\iv)T(a\iv)\iv\Big).
\end{align*}
In order to verify that the last two determinants are the same, use formula \eqref{f.TivH} twice and also employ the general formula
$\det (I+AB)=\det (I+BA)$, which holds for bounded Hilbert space operators $A$ and $B$ whenever
both $AB$ and $BA$ are trace class.
\end{proof}

Our next goal is to obtain the following generalization of the previous determinant identity, 
\begin{equation}\label{det-det-?}
\det T(\tp)T(\tp_R)^{-1} \cdots T(\tp_0)^{-1} = 
\det T(\phi_0^{-1})^{-1}\cdots T(\phi_R^{-1})^{-1}T(\phi^{-1}),
\end{equation}
where $\phi=\phi_0\phi_1\cdots \phi_R$.
Unfortunately, we are only able to prove this identity under assumptions which are stronger than those
one would suspect to be necessary. This raises the question on the range of validity of this identity
and makes further investigations desirable.

Despite its limitations and its cumbersome formulation, the following result will be just sufficient for our purpose of
dealing with the piecewise continuous case and applying it to the determinants that occur in Theorem \ref{t:asym:1b}.

\begin{theorem}\label{id:det-R}
Let $\phi_0\in L^\iy(\T)\NN$, $\sigma_1,\dots,\sigma_R\in L^\iy(\T)\NN$, and put
$$
\phi_{k,\la}=\exp(\la \sigma_k), \qquad 1\le k\le R.
$$
Let $U_1,\dots,U_R\subseteq \C$ be open connected subsets containing the origin and assume that 
\begin{itemize}
\item[(i)]
$T(\tp_0)$ is invertible on $(\ell^2)^N$;
\item[(ii)]
$T(\tp_{j,\la})$ are invertible  on $(\ell^2)^N$ for $\la\in U_j$ and each $1\le j\le R$;
\item[(iii)]
the operators $H(\ta)H(b)$ and $H(b^{-1})H(\ta^{-1})$ are trace class whenever
$$
a=\phi_0\cdot \phi_{1,\la_1}\cdot \phi_{2,\la_2}\,\cdots\,\phi_{k-1,\la_{k-1}},\qquad b=\phi_{k,\la_k}
$$
where $\la_j\in U_j$ for $1\le j\le k\le R$.
\end{itemize}
Then the following two operator determinants,
\begin{align*}
f_R(\la_1,\dots,\la_R) &=\det T(\tp_0 \tp_{1,\la_1}\cdots\,\tp_{R,\la_{R}}) T(\tp_{R,\la_R})\iv \cdots T(\tp_{1,\la_1})\iv T(\tp_0)\iv,
\\
g_R(\la_1,\dots,\la_R) &=\det 
T(\phi_0\iv)\iv T(\phi_{1,\la_1}\iv)\iv\cdots T(\phi_{R,\la_R}\iv)\iv 
T(\phi_{R,\la_R}\iv \cdots \phi_{1,\la_1}\iv \phi_0\iv),
\end{align*}
are well-defined (for $\la_j\in U_j$, $1\le j\le R$) and coincide.
\end{theorem}
\begin{proof}
Let us first note that there is nothing to prove for $R=0$. Moreover, for $R=1$, the assertions follow directly from 
Proposition \ref{id:det-2}.

Next notice that the inverse involved in these two operator determinants exist.
This is due to assumptions (i) and (ii) and the application of  Lemma \ref{lem.TinvH}.

In order to see that the determinants are well defined we need to confirm that the underlying operators 
are of the form identity plus trace class. Equivalently, for the first determinant, 
$$
T(\tp_0 \tp_{1,\la_1}\cdots\tp_{R,\la_{R}}) = T(\tp_0) T(\tp_{1,\la_1})\cdots\,T(\tp_{R,\la_{R}}) +
\mbox{trace class}.
$$
This can be shown by induction on $R$, where the case $R=1$ is settled by assumption (iii) with $k=1$.
Assuming $R\ge 2$, we can pass from $R-1$ to $R$ by writing
\begin{align*}
T(\tp_0 \tp_{1,\la_1}\cdots\tp_{R,\la_{R}}) 
&= T(\tp_0 \tp_{1,\la_1}\cdots\tp_{R,\la_{R-1}})T(\tp_{R,\la_{R}})
\\
&\qquad +H(\tp_0 \tp_{1,\la_1}\cdots\tp_{R,\la_{R-1}})
H(\phi_{R,\la_R})
\end{align*}
where we used \eqref{T1}. The product of the Hankel operators is trace class by assumption (iii) with $k=R$.
Now apply the induction hypothesis to the term $T(\tp_0 \tp_{1,\la_1}\cdots\tp_{R,\la_{R-1}})$ to see that it 
is $T(\tp_0) T( \tp_{1,\la_1})\cdots T(\tp_{R,\la_{R-1}})$ plus trace class. 
This concludes the proof that the first operator determinant is well-defined, and the arguments for the second one
are similar.

It remains to show that 
$$
f_R(\la_1,\dots,\la_R)=g_R(\la_1,\dots,\la_R)
$$
for $(\la_1,\dots,\la_R)\in U_1\times \dots\times U_R$.
We notice that the functions $f_R(\la_1,\dots,\la_R)$ and $g_R(\la_1,\dots,\la_R)$
depend analytically on each of the variables $\la_1\in U_1,\dots, \la_R\in U_R$.
More specifically we have analyticity in any one of these variables while keeping the others fixed.
Because the sets $U_1,\dots, U_R$ are open and connected and contain the origin, it is sufficient to show that the two functions coincide when
$(\la_1,\dots,\la_R)\in D_\eps\times \dots \times D_\eps$ where 
$D_\eps=\{\la\in\C: |z|<\eps\}$ and $\eps>0$ can be arbitrarily small.

Under the assumption that $|\la_k|<\eps$ (and having chosen $\eps$ sufficiently small) we can assume that the operators
$$
T(\tp_0\tp_{1,\la_1}\cdots \tp_{k,\la_{k}})
$$
and
$$
T(\phi_{k,\la_{k}}\iv\cdots \phi_{1,\la_1}\iv\phi_0\iv)
$$
(with $1\le k\le R$) are invertible since they  are sufficiently close to the operators $T(\tp_0)$ and $T(\phi_0\iv)$  in operator norm,
and the latter are assumed to be invertible.
This enables us to prove the identity via induction. Indeed, we can split
\begin{align*}
\lefteqn{\det  T(\tp_0 \tp_{1,\la_1}\cdots\tp_{R,\la_{R}}) T(\tp_{R,\la_R})\iv \cdots T(\tp_{1,\la_1})\iv T(\tp_0)\iv}\qquad
&\\[1ex]
&= \det T(\tp_0 \tp_{1,\la_1}\cdots\tp_{R,\la_{R}}) T(\tp_{R,\la_R})\iv T(\tp_0 \tp_{1,\la_1}\cdots \tp_{R-1,\la_{R-1}})\iv
\\
&\qquad\times
\det T(\tp_0 \tp_{1,\la_1}\cdots\,\tp_{R-1,\la_{R-1}}) T(\tp_{R-1,\la_R-1})\iv\cdots T(\tp_{1,\la_1})\iv T(\tp_0)\iv,
\\[2ex]
\lefteqn{\det 
T(\phi_0\iv)\iv T(\phi_{1,\la_1}\iv)\iv\cdots T(\phi_{R,\la_R}\iv)\iv 
T(\phi_{R,\la_R}\iv \cdots \phi_{1,\la_1}\iv \phi_0\iv)}\qquad
&\\[1ex]
&= 
\det T(\phi_0\iv)\iv  T(\phi_{1,\la_1}\iv)\iv\cdots T(\phi_{R-1,\la_{R-1}}\iv)\iv 
T(\phi_{R-1,\la_{R-1}}\iv \cdots \phi_{1,\la_1}\iv \phi_0\iv)
\\
&\qquad\times 
\det T(\phi_{R-1,\la_{R-1}}\iv \cdots \phi_{1,\la_1}\iv \phi_0\iv)\iv  T(\phi_{R,\la_R}\iv)\iv 
T(\phi_{R,\la_R}\iv \cdots \phi_{1,\la_1}\iv \phi_0\iv).
\end{align*}
Invoking Proposition \ref{id:det-2} and the induction hypothesis proves the identity.
\end{proof}

\begin{corollary}\label{c4.9}
Assume $\phi=\phi_0\phi_1\cdots \phi_R$ where $\phi_k=u_{B_k,\tau_k}$, $1\le k\le R$, and such that the eigenvalues of $B_{k}\in\C^{N\times N}$ have real parts in the interval $I = (-1/2, 1/2).$ Moreover, suppose also that $\phi_0\in \CpwTG^{N\times N}$ is such that both $T(\phi_0)$ and $T(\tilde{\phi}_0)$ are invertible on $(\ell^2)^N$.
Then 
\begin{align}\label{f.asym3}
\lim_{n\to\infty} 
\frac{\det T_n(\phi)}{G[\phi_0]^n \prod_{k=1}^R \det T_n(\phi_k)}
=E
\end{align}
where
\begin{align}\label{E.det}
E = \det T(\phi) T\iv(\phi_R) \cdots T\iv(\phi_1)
T\iv(\phi_1\iv)\cdots T\iv(\phi_R\iv) T(\phi\iv).
\end{align}
\end{corollary}
\begin{proof}
We have the same assumptions as in Theorem \ref{t:asym:1b}. 
Below we will argue that we can apply Theorem \ref{id:det-R} and conclude 
that the constant 
$$
E_2=  \det T(\tilde{\phi}) T(\tilde{\phi}_R)^{-1}\cdots T(\tilde{\phi}_0)^{-1},
$$
evaluates to 
$$
E_2=\det T(\phi_0\iv)\iv T(\phi_1\iv)\iv \cdots T(\phi_R\iv)\iv T(\phi\iv).
$$
Taking this for granted we can combine the constants $E_1, E_2$, and $E_3$ in the proper order, observe the cancellation of
$T(\phi_0)$ and $T(\phi_0\iv)$, and arrive at the expression for $E$. 

In order to see that Theorem \ref{id:det-R} can be applied we put
$$
\phi_{k,\la}=\exp(\la \sigma_k)\quad \mbox{with}\quad
\sigma_k(t)=B_k \log(-t/\tau_k)
$$
where $\sigma_k$ is a piecewise continuous matrix function, which is continuous on $\T\setminus \{\tau_k\}$ 
with normalization $\sigma_k(-\tau_k)=0$. It follows that 
$$
\phi_{k,\la}=u_{\la B_k,\tau_k}
$$
so that for $\la=1$ this function coincides with $\phi_k$, $1\le k\le R$. Using Lemma \ref{l.traceHH}
we see that the Hankel conditions in (iii) of Theorem  \ref{id:det-R} are fulfilled.

We also note that besides (i), also (ii) is satisfied, i.e., the operators
$T(\tp_{k,\la})$ are invertible for $\la \in U_k$ when we choose $U_k$ to be a sufficiently small open neighborhood
of the interval $[0,1]\subseteq \C$. Indeed, for $\la\in [0,1]$, the eigenvalues of $\la B_k$
have real parts also lying in $(-1/2,1/2)$ and therefore $T(\tp_{k,\la})=T(\wt{u_{\la B_k,\tau_k}})$
is invertible by Proposition \ref{p.pure}. Clearly, invertibility is guaranteed in a small neighborhood of $[0,1]$ as well.
Thus all assumptions in Theorem  \ref{id:det-R} are fulfilled.
\end{proof}

\bigskip

Notice that the operator determinant $E$ is well-defined even if $T(\phi_0)$ or $T(\phi_0^{-1})$ are not invertible
(however, still assuming all other assumptions in the corollary).
Indeed, it follows from the first part of Proposition \ref{p3.1} that 
\begin{align*}
T(\phi) T(\phi_R)\iv \cdots T(\phi_1)\iv
&= T(\phi_0)+\mbox{trace class},
\\
T(\phi_1\iv)\iv\cdots T(\phi_R\iv)\iv T(\phi\iv)
&= T(\phi_0\iv)+\mbox{trace class}.
\end{align*}
Hence the product is identity plus trace class since also $T(\phi_0)T(\phi_0^{-1})$, the operator featuring in $E_3$, is identity plus trace class.

It is the goal of the next section to remove the extra assumption that $T(\phi_0)$ and $T(\phi_0^{-1})$ be invertible
and replace it by a weaker condition on $\phi_0$.
Once this is accomplished our main results,  Theorem \ref{main-3} and Corollary \ref{c.main-2}, can be proved.

\section{The final derivation of the asymptotics}\label{final results}

\subsection{Some auxilliary results}

We will need an auxiliary result which provides a product representation for $\phi_0\in\CpwTG\NN$
 in terms of exponentials. To prove this result some information about the commutative Banach algebra
$\CpwTG$ is required.

\begin{lemma}\label{l5.0}
$\cB=\CpwTG$ is a commutative Banach algebra
with the following properties:
\begin{enumerate}
\item[(a)]
$\cB$ is continuously embedded and inverse closed in $C(\T)$,
\item[(b)]
$\cB$ contains all functions in $C^\infty(\T)$,
\item [(c)]
the maximal ideal space of $\cB$ is (naturally) homeomorphic to $\T$.
\end{enumerate}
\end{lemma}
\begin{proof}
Properties (a) and (b) are obvious from the definition. Hence we will focus on proving (c).
We note that (c) implies that the Gelfand transform $\cB\to C(\cM(\cB))$ amounts to the natural embedding of $\cB$ into $C(\T)$.

To prove (c), consider the map
$$
\Lambda:t_0\in \T\mapsto \Phi_{t_0}\in \cM(\cB)
$$
where $\Phi_{t_0}$ is the multiplicative linear functional defined by $\Phi_{t_0}(b)=b(t_0)$, $b\in \cB$.
This map is well-defined.

We claim that $\Lambda$ is surjective.
Indeed, let $\Phi$ be a multiplicative linear functional on $\cB$. Apply it to the function $\chi_1(t)=t$ to determine a number 
$t_0:=\Phi(\chi_1)$. By Gelfand theory, the value $t_0$ is contained in the spectrum of $\chi_1$ considered as an element in $\cB$.
Hence $t_0\in\T$. We will show that $\Phi=\Phi_{t_0}$.

Consider an arbitray $b\in \cB$. For $\eps>0$, choose a $C^\infty$-function $f_\varepsilon:\T\to [0,1]$ which is equal to one in an $\varepsilon$-neighborhood of
$t_0$ and vanishes outside a $2\varepsilon$-neighborhood of $t_0$. Write
$$
b(t)-b(t_0) = f_\varepsilon(t)(b(t)-b(t_0))+ \frac{1-f_\varepsilon(t)}{t-t_0}(t-t_0) (b(t)-b(t_0)).
$$
All functions involved in the product on the right hand side belong to $\cB$. 
In particular, $\Phi(t-t_0)=\Phi(\chi_1)-t_0 \Phi(1)=0$, and thus we get
$$
\Phi(b)-\Phi_{t_0}(b) = \Phi\Big( f_\varepsilon(b-b(t_0))\Big)\in 
\mathrm{sp}_{\cB}\Big( f_\varepsilon(b-b(t_0))\Big)=
\im \Big( f_\varepsilon(b-b(t_0))\Big).
$$
Notice that the spectrum of a function considered as an element in $\cB$ is equal  to the image of this function due to the inverse closedness
stated in part (a). 
As we can make $\varepsilon>0$ as small as desired it follows from the continuity of $b(t)$ that $\Phi(b)=\Phi_{t_0}(b)$. This holds for all $b\in \cB$, and therefore 
$\Phi=\Phi_{t_0}$ and $\Lambda$ is surjective.

The injectivity of $\Lambda$ can be seen by applying $\Phi_{t_0}=\Phi_{t_1}$ to the function $\chi_1(t)=t$ to conclude that $t_0=t_1$.

Finally, let us prove that $\Lambda$ is a homeomorphism. The standard local base for the topology of $ \cM(\cB)$ at an element $\Phi\in \cM(\cB)$ consists
of all neighborhoods
$$
U_{b_1,\dots,b_n;\varepsilon}[\Phi]=\Big\{\Psi\in \cM(\cB)\,:\, |\Psi(b_i)-\Phi(b_i)|<\varepsilon,\;\;1\le i\le n\Big\}.
$$
For $\Phi=\Phi_{t_0}=\Lambda(t_0)$, $t_0\in\T$, the pre-image equals
$$
\Lambda^{-1}(U_{b_1,\dots,b_n;\varepsilon}[\Phi])
=\Big\{t\in\T\,:\, |b_i(t)-b_i(t_0)|<\varepsilon,\;\;1\le i\le n\Big\},
$$
which contains a small $\delta$-neighborhood of $t_0\in\T$. With this, the continuity of $\Lambda$ at each $t_0\in\T$ is proved.
To see that $\Lambda$ is an open map, one could either use the particular neighborhood $U_{\chi_1;\varepsilon}[\Phi]$, or
invoke the compactness and Hausdorff properties of the underlying spaces.
Thus it is proved that $\Lambda$ is a homeomorphism.
\end{proof}

\begin{lemma}\label{l5.1}
Let $\phi_0\in \CpwTG^{N\times N}$ be an invertible function and suppose that $\wind (\det \phi_0)=0$. Then, for each $\varepsilon>0$,
the function $\phi_0$ admits a product representation
\begin{equation}\label{eq:rep}
\phi_0=e^{\eta_0}e^{\eta_1}\cdots e^{\eta_S},
\end{equation}
where $\eta_k\in \CpwTG^{N\times N}$ and $\|\eta_k\|<\eps$ for all $0\le k\le S$.
\end{lemma}

In the last condition, the norm of  $\CpwTG^{N\times N}$ can be considered. 

Because of its importance we will give two proofs of this lemma. One is based on Wiener-Hopf factorization and is in principle constructive, while the other is based on Gelfand theory and uses Arens Theorem (a generalization of the Arens-Royden Theorem) and a homotopy argument. 
For the notion of Wiener-Hopf factorization, we refer to the monographs \cite{CG, LS}, the survey paper \cite{GKS}, and the references therein.

\begin{proof}
Let us first notice that the norm condition $\|\eta_k\|<\eps$ can be disregarded. Indeed, if this condition is not fulfilled we can replace a single exponential $e^{\eta_k}$ in the product representation by a product of $M$ exponentials
$$
e^{\eta_k}=e^{\eta_k/M}\cdots e^{\eta_k/M}
$$
and choose $M$ sufficiently large. This comes only at the expense of the number of factors involved
in the product representation.

{}From now on we will use the notation $\cB=\CpwTG$. It is easy to see that inverse closedness of $\cB$ in $C(\T)$
immediately implies the inverse closedness of $\cB\NN$ in $C(\T)\NN$. Therefore, the assumption that $\phi_0\in\cB\NN$ is invertible implies that $\phi_0^{-1}\in\cB\NN$.

{\bf First argument:}\ \ 
The continuous matrix function $\phi_0$ can be approximated as closely as desired in the $L^\infty$-norm by a
matrix Laurent polynomial, e.g., by one that is obtained from the Fourier series of $\phi_0$. Considering a sufficiently close approximation
by a matrix Laurent polynomial $b(t)$, one for which 
$$
\|\phi_0-b\|_{L^\infty}<\frac{1}{2\|\phi_0^{-1}\|_{L^\iy}},
$$
we will obtain a representation
\begin{equation}\label{eq:rep0}
\phi_0=e^{\eta_0}b.
\end{equation}
Indeed, to see this put
$$
\varrho=(b-\phi_0)\phi_0^{-1}=b \phi_0^{-1}-I_N
$$
and notice that $\varrho$ is a continuous matrix function with $\|\varrho\|_{L^\infty}<1/2$. Therefore $I_N+\varrho$ has 
a continuous logarithm and we can define
$$
\eta_0= -\log(I_N +\varrho).
$$
Combining this we obtain $e^{-\eta_0}=I_N+\rho=b\phi_0^{-1}$ and the representation \eqref{eq:rep0} follows.

{}From the definition of $\varrho$ it is immediate that $\varrho\in \cB\NN$. Because of the inverse closedness of $\cB\NN$ in $C(\T)\NN$ (and thus in $L^\iy(\T)\NN$) we see that the spectrum 
$$
\mathrm{sp}_{\cB\NN}(\varrho) = \mathrm{sp}_{L^\iy(\T)\NN}(\varrho) \subseteq \{z\in\C:\,|z|<1/2\}.
$$
Hence the above logarithm can also be expressed using Riesz functional calculus,
$$
\eta_0=-\frac{1}{2\pi i}\oint_{|z|=1/2}(zI_N-\varrho)^{-1} \log(1+z)\, dz,
$$
and this entails that $\eta_0\in \cB\NN$.

To summarize, at this point we have extracted the first factor in the desired product representation \eqref{eq:rep},
and  we are left with representing $b$ as a finite product of exponentials. 

Before we are going to do this we notice that taking the determinants in \eqref{eq:rep0} gives
$$
\det \phi_0(t)= e^{\tr\eta_0(t)}\det b(t),
$$
where the exponential is a function with winding number zero. Therefore, we conclude that the matrix Laurent polynomial $b(t)$
not only takes invertible values on $\T$ (and thus on some  open neighborhood of $\T$), but that the winding number of $\det b(t)$ is zero. 

A well-known factorization result for matrix functions (see~\cite[Thm.~5.5 and its remark]{BS06}, \cite{CG}, or \cite{LS}) implies that under these 
conditions on $b$, the function admits a Wiener-Hopf factorization 
$$
b(t)=b_-(t) d(t) b_+(t),\qquad t\in \T,
$$
not necessarily canonical, but with $d(t)=\diag[t^{\kappa_1},\dots,t^{\kappa_N}]$ 
where $\kappa_1,\dots,\kappa_N\in\Z$ are the partial indices of the factorization. 
The winding number condition implies that 
$\kappa_1+\dots+\kappa_N=0$.
The factors $b_+(t)$ and $b_-(t)$ are also matrix Laurent polynomials, which along with their inverses are analytic
 on the sets
$\{ z\in\C \,:\, |z|<1+\delta\}$ and $\{\infty\}\cup\{z\in\C\,:\, |z|>1-\delta\}$ for some sufficiently small $\delta>0$.

By continuously deforming $b_+(t)$ to $b_+(0)$ we can represent $b_+(t)$ as a finite product
involving factors which are close to the identity $I_N$ in some sense,
$$
b_+(t)=b_+(0)\prod_{j=1}^M b_+(r_{j-1} t)^{-1} b_+(r_jt),\qquad r_j=\frac{j}{M}, \qquad t\in\T,
$$
where $M$ is chosen sufficiently large. The constant invertible matrix $b_+(0)$ has a matrix logarithm and thus itself is an exponential.
The factors
$$
f_j(t)=b_+(r_{j-1}t)\iv b_+(r_jt)
$$
are analytic (and close to $I_N$) on some neighborhood of $\T$. Therefore they have an analytic matrix logarithm $\log f_j(t)$, which 
when considered as function on $\T$ belongs to $\cB\NN$. We have proved that $b_+(t)$ is the finite product of exponentials
of functions in $\cB\NN$.

The deformation argument can be applied in the similar manner to the factor $b_-(t)$ by deforming it to the constant invertible matrix $b_-(\infty)$. 
Thus also $b_-(t)$ admits a representation as a finite product of exponentials of functions in $\cB\NN$.

It remains to show that $d(t)$   also has  such a product representation.
In case all $\kappa_k$ are equal to zero nothing needs to be done since $d(t)=I_N$. 
Otherwise, using the condition  $\kappa_1+\dots+\kappa_N=0$ it follows easily by induction
that $d(t)$ can be written as a finite product of diagonal matrices,
where each of these diagonal matrices has $t$ and $t^{-1}$ precisely once as an
entry and otherwise $1$ on the remaining diagonal entries. We can formally write
$$
d(t)=d_1(t)\cdots d_M(t)
$$
with
$$
d_k(t)=P_k \left(\ba{ccc} t & 0 & 0 \\ 0 & t^{-1} & 0 \\ 0 & 0 & I_{N-2} \ea\right) P_k^{-1}
$$
and $P_k$ being permutation matrices. Focusing on the $2\times 2$ part of the term in the middle,
we decompose it into
$$
 \left(\ba{ccc} t & 0  \\ 0 & t^{-1} \ea\right)
 =
 \left(\ba{ccc} 1 & 0  \\ t^{-1} & 1 \ea\right)
 \left(\ba{ccc} 0& 1  \\ -1 & 0 \ea\right)
\left(\ba{ccc} 1 & 0  \\ t & 1 \ea\right)
 \left(\ba{ccc} 1 & -t^{-1}  \\ 0 & 1 \ea\right).
$$
Each of these factors is the exponential of either a constant matrix or a very simple matrix Laurent polynomial.
This gives rise to a corresponding product representation of each of the matrix functions
$d_k(t)$ by making an obvious extension to the full $N\times N$ matrices and taking the permutation matrices into account. 
It follows that the factor $d(t)$ also has a representation as  a product of exponentials.
This concludes our first proof.

{\bf  Second argument:}\ \
As shown in Lemma \ref{l5.0}, the maximal ideal space of the commutative Banach algebra
$\cB$ is homeomorphic to $\T$. In fact, the Gelfand transform 
$a\in\cB\mapsto \hat{a}\in C(\mathcal{M}(\cB))$ amounts to the natural embedding of $\cB$ into $C(\T)$.

Let $\cG_1(\cB\NN)$ stand for the connected component of the group of all invertible elements 
in the Banach algebra $\cB\NN$ containing the identity.
It is known \cite[Thm.~10.34]{Ru} that $\phi_0$ admits a representation \eqref{eq:rep} if and only if $\phi_0$ belongs to 
$\cG_1(\cB\NN)$.
Furthermore, using Gelfand theory,  Arens Theorem \cite{Arens,Tay}  implies that 
$\phi_0\in \cG_1 (\cB\NN)$ if and only if $\phi_0\in \cG_1 (C(\T)\NN)$. 

Since by assumption we know that $\phi_0$ is a continuous matrix function on $\T$ taking invertible values, it follows that 
the last condition $\phi_0\in \cG_1 (C(\T)\NN)$ can be rephrased  by saying that the mapping 
$$
\phi_0:\T \to GL(N;\C)
$$
is homotopic to the constant mapping (with value $I_N$). Here we use the notation $GL(N;\C)$ for the 
general linear group of order $N$ over $\C$, i.e., the group of nonsingular $N\times N$ complex matrices. 
{}From homotopy theory it is well-known that $GL(N;\C)$ is (path) connected and that its
first homotopy group $\pi_1(GL(N;\C))$ is isomorphic to $\Z$ by means of the mapping
$$
[b]_\sim \in \pi_1(GL(N;\C)) \mapsto \wind(\det b) \in \Z.
$$
Therefore, since the $\wind(\det \phi_0))=0$ is assumed, we can conclude that $\phi_0$ is homotopic to 
the constant map. In other words, $\phi_0\in \cG_1 (C(\T)\NN)$. Due to the equivalencies stated above, 
we thus have proved that $\phi_0$ admits the product representation 
\eqref{eq:rep}.
\end{proof}

\subsection{Relaxation of the invertibility assumption.}

We can now use the previous lemma to prove the following theorem, which is almost identical to Corollary~\ref{c4.9} except that the assumption about the invertibility of $T(\phi_0)$ and $T(\wt{\phi}_0)$ has been replaced by the weaker condition on winding number of $\det \phi_0$.

\begin{theorem}\label{main-1}
Assume $\phi=\phi_0\phi_1\cdots \phi_R$ where $\phi_k=u_{B_k,\tau_k}$, $1\le k\le R$, and such that the eigenvalues of $B_{k}\in\C^{N\times N}$ have real parts in the interval $I = (-1/2, 1/2).$ Moreover, suppose that $\phi_0\in \CpwTG^{N\times N}$ is invertible and 
$\wind(\det \phi_0)=0$.
Then 
\begin{align}\label{f.asym1b}
\lim_{n\to\infty} 
\frac{\det T_n(\phi)}{G[\phi_0]^n \prod_{k=1}^R \det T_n(\phi_k)}
=E
\end{align}
where
$$
E = \det T(\phi) T(\phi_R)\iv \cdots T(\phi_1)\iv 
T(\phi_1\iv)\iv \cdots T(\phi_R\iv)\iv T(\phi\iv)
$$
is a well-defined operator determinant.
\end{theorem}

\begin{proof}
The proof goes along the lines of the proofs of Theorem \ref{t:asym:1b} and Corollary \ref{c4.9} except that we refine the given product representation
$\phi=\phi_0\phi_1\cdots \phi_R$
to
$$
\phi=\psi_0\psi_1\cdots \psi_S\phi_1\cdots \phi_R.
$$
This is justified by Lemma \ref{l5.1}, which allows us to write $\phi_0=\psi_0\psi_1\cdots\psi_S$
with $\psi_j=e^{\eta_j}\in \CpwTG\NN$ and $\eta_j\in\CpwTG\NN$ for $0\le j\le S$.
We may also assume that $\|\eta_k\|_{L^\iy}<1/2$. The latter condition implies that 
$\|\psi_k-I_N\|_{L^\infty}<1$,
and therefore the operators $T(\psi_j)$ and $T(\widetilde{\psi}_j)$ are invertible and the sequences $T_n(\psi_j)$ are stable ($0\le j\le S$).

In the proof of Theorem \ref{t:asym:1b} we have used Theorem \ref{t:asym:1}. We will also use it  here 
and apply it to the larger product. To be specific, we consider 
$$
A_n= T_n(\phi)T_n(\phi_R)^{-1}\cdots T_n(\phi_{1})^{-1}T_n(\psi_S)^{-1}\cdots T_n(\psi_1)^{-1}  T_n(\psi_0)^{-1},
$$
which leads us to 
$$
\lim_{n\to\infty}\frac{\det T_n(\phi)}{\prod_{j=0}^S G[\psi_j]^n \prod_{k=1}^R\det T_n(\phi_k)}=E
$$
with $E=E_1E_2E_3$ and the constants
\begin{align*}
E_1 &= \det T(\phi) T(\phi_R)^{-1}\cdots T(\phi_1)^{-1} T(\psi_S)^{-1}\cdots T(\psi_1)^{-1} T(\psi_0)^{-1}
\\
E_2 &= \det T(\tilde{\phi}) T(\tilde{\phi}_R)^{-1}\cdots T(\tilde{\phi}_1)^{-1}
T(\tilde{\psi}_S)^{-1}\cdots T(\tilde{\psi}_1)^{-1} T(\tilde{\psi}_0)^{-1}
\\
E_3 &= \prod_{j=0}^S\det T(\psi_j)T(\psi_j^{-1}).
\end{align*}
Here $E_3$ and the $G[\psi_j]$'s arise from the Szeg\H{o}-Widom limit theorem applied
to each of the determinants $\det T_n(\psi_j)$.

We notice that the definition \eqref{const.G} of the constants $G[\,\cdot\,]$ implies that 
$$
G[\phi_0]=G[\psi_0]G[\psi_1]\cdots G[\psi_S].
$$
Indeed, this is deduced from $\phi_0=\psi_0\psi_1\cdots \psi_S=e^{\eta_0}e^{\eta_1}\cdots e^{\eta_S}$.

The next step is to rewrite $E_2$ as
$$
E_2=\det T(\psi_0\iv)\iv T(\psi_1\iv)\iv \cdots T(\psi_S\iv)\iv T(\phi_1\iv)\iv \cdots T(\phi_R\iv)\iv T(\phi\iv).
$$
This follows directly from Theorem \ref{id:det-R}. To make the connection, we proceed similarly as in 
Corollary \ref{c4.9}, but modify the notation used in 
denoting the appropriate products to the following,
$$
\psi_0 \psi_{1,\mu_1} \psi_{2,\mu_2}\cdots \psi_{S,\mu_S}\phi_{1,\lambda_1}\dots \phi_{R,\lambda_R},
$$
where 
$$
\psi_{j,\mu}(t)=e^{\mu \eta_k(t)}\,\qquad 
\phi_{k,\lambda}(t)=e^{\lambda \sigma_k(t)}=u_{\lambda B_k,\tau_k}.
$$
In other words we introduce the parameters $\mu_1,\dots,\mu_S,\lambda_1,\dots,\lambda_R$.
Note that it is not necessary to parameterize $\psi_0$. As before, the assumption
$\|\eta_j\|_{L^\infty}<1/2$ implies that $\|\psi_{j,\mu}-I_N\|_{L^\iy}<1$ for all $\mu\in [0,1]$,
and hence for all $\mu$ in a small neighborhood $U_j$ of $[0,1]$. Therefore,
all the Toeplitz operators $T(\wt{\psi}_{j,\mu})$ are invertible whenever $\mu\in U_j$.
With this we see that Theorem \ref{id:det-R} can be applied and the above expression for $E_2$ follows.

By a simple computation it can be seen that 
$E=E_1E_2E_3$ evaluates to the constant given in the theorem. Indeed, we write $E_1 E_2$ as single operator determinant as follows
\begin{align*}
&\det \Big(
T(\psi_0\iv)\iv T(\psi_1\iv)\iv \cdots T(\psi_S\iv)\iv T(\phi_1\iv)\iv \cdots T(\phi_R\iv)\iv T(\phi\iv)
\\
&\qquad\times
T(\phi) T(\phi_R)^{-1}\cdots T(\phi_1)^{-1} T(\psi_S)^{-1}\cdots T(\psi_1)^{-1} T(\psi_0)^{-1}
\Big)
\end{align*}
and combine it recursively with the other operator determinants $\det  T(\psi_j)T(\psi_j^{-1})$ for $
0\le j\le S$. All we have to use here are the general formulas $\det T A T^{-1}=\det A$ and $\det (AB) =(\det A) (\det B)$
where $A$ and $B$ are identity plus trace class and $T$ is an invertible operator.
\end{proof}

\subsection{Proof of the main results.}

We are now going to prove our main results, Theorem \ref{main-3} and Corollary \ref{c.main-2}.

We first remark that we have already proved Theorem \ref{thm:index} (see Section \ref{index proof}),
which states  the equivalence of the assumptions (i)--(iv). Therefore, let us assume that 
$\phi\in PC^{1+\eps}(\T;\Gamma)\NN$ is $I$-regular and $\wind(\phi;I)=0$ where
$I=(-1/2,1/2)$. We now apply Proposition \ref{p.prod} and Corollary \ref{p.prod.cor} to see that $\phi$ admits a representation \eqref{f.prod1}.
That means we can write $\phi=\phi_0\phi_1\cdots \phi_R$ with $\phi_k=u_{B_k,\tau_k}$ for $1\le k\le R$, and 
$\phi_0\in \CpwTG^{N\times N}$ is an invertible function with $\wind(\det \phi_0)=0$. The real parts of the eigenvalues of $B_k$ lie in $I$.
This proves the first part of Theorem \ref{main-3}.

The conditions which we just stated are the assumptions in Theorem \ref{main-1}. Therefore, we can apply this theorem and conclude the asymptotics
\eqref{f.asym1b}. In this connection we will also use what we stated in \eqref{f4.3}--\eqref{f4.5} (which in turn followed from Proposition \ref{p.pure}).
Combining all this we arrive at the asymptotic formula \eqref{e:asymp} with the constants given by \eqref{G.con1}, \eqref{Om.con1}, and \eqref{f.constE1}.
Hence Theorem  \ref{main-3} is proved.

 Corollary \ref{c.main-2} is a direct conclusion of Theorem \ref{main-3}. In view of the product representation $\phi=\phi_0\phi_1\cdots\phi_R$
 and the definition of the function $c$ in \eqref{fct.c}, we see that $c=\det \phi_0$ and thus \eqref{G.con0} follows. 
 On the other hand, the matrices $B_k$ are similar to the matrices $L_k$ defined in \eqref{f.Lk}. This can been seen directly, but has already been noted 
 in the paragraphs following Proposition \ref{p.prod} and its proof. This similarity implies formula \eqref{Om.con0}.

Finally, the constant $E$ is nonzero if and only if the corresponding operator determinant is nonzero.
Notice that the Barnes G-functions are all nonzero since the $\beta$'s have real parts in $(-1/2,1/2)$ and thus are
not nonzero integers. Our assumptions imply that both $T(\phi)$ and $T(\phi^{-1})$
are Fredholm operators with index zero (see Theorem \ref{thm:index} and Proposition \ref{p.T-Fred}). 
The inverses of $T(\phi_k)$ and $T(\phi_k^{-1})$ exist. Note that an operator determinant is nonzero
if and only if the underlying operator has a trivial kernel and a trivial cokernel. Therefore, it is easily seen that this amounts to the invertibility of both 
$T(\phi)$ and $T(\phi^{-1})$. This concludes the proof of Corollary \ref{c.main-2}.
\hfill $\Box$

\subsection{An alternate approach using perturbation results and further open problems}\label{alt-approach}

As we have mentioned before, many of the difficulties in our work arise because the invertibility of block Toeplitz operators cannot be guaranteed. The proof of the Szeg\H{o}-Widom limit theorem given in \cite{Wi76} gets around this problem by using an elegant perturbation result \cite{Wi75a}. Indeed, if $T(\phi)$ is a  Fredholm operator with index zero one can find a matrix Laurent polynomial $q$ such that $T(\phi+\lambda q)$ is invertible whenever $0<|\lambda|<\eps$. It would be interesting to know whether a proof of our result could also be given using this type of perturbation. This is very likely possible, but notice that because it has been used in connection with localization results, it is probably not as simple as in \cite{Wi76}. In particular, it seems that one would need a result to simultaneously perturb  $T(\phi)$ and $T(\wt{\phi})$ with the same $q$ to make them both invertible. Such a result has not yet been established.

Let us mention some open questions that naturally come up. The first one is about the validity of the
`duality' formula \eqref{det-det-?}. Clearly, this formula can be proved if all `intermediate' Toeplitz operators
are invertible. It can also be proved if the various symbols can be continuously deformed to ones for which the identity holds
as long as the Toeplitz operators with the deformed symbols are invertible. However, results such as \cite[Prop.~10.5]{E03}
suggest that the deformation argument has its limitations. Perhaps a perturbation argument could work again, but it seems that 
if \eqref{det-det-?} turns out to be generally true, its proof is not simply algebraic like the one for \eqref{det-det-2}.

Another question of further interest is about the precise smoothness (on $\phi_0$) that is required to ensure that the stated Toeplitz determinant asymptotics \eqref{e:asymp}
is valid in the case of piecewise continuous symbols $\phi$. As far as we are know, this has not yet been examined even in the scalar case ($N=1$).
However, in the case of continuous selfadjoint symbols  (i.e., for the Szeg\"o-Widom theorem) results of this kind are available. We refer to  \cite[Sect.~10.8]{BS06} for more details and references.

\section{Examples}
\label{sec:Examples}

In this section we study six examples of matrix-valued discontinuous symbols using our results. We first present two examples that illustrate some of the subtleties of considering matrix-valued symbols. 

\begin{example}
Observe that the finite Toeplitz matrices with symbols of the simple form
\[ \twotwo{f}{g}{0}{h}\,\,\,\, \mbox{or} \,\,\,\, \twotwo{f}{0}{0}{h}\] 
have exactly the same determinant asymptotics. This can be easily seen by simply rearranging rows and columns. 
Thus, even if $g$ is piecewise continuous, the discontinuity will not contribute to the asymptotics. 
\end{example}

\begin{example}
For a second example, consider 

\[ \phi = \twotwo{u_{\beta,1}}{b}{c}{u_{\beta,1}} \] 

where $b, c\in\C$ and $\beta\in (-1/2,1/2)$ are constants chosen so that $\phi$ is invertible. Using our recipe for finding the appropriate $u_{B}$ in our canonical representation, we need to compute Jordan form of 
\[ (\phi(1+0))^{-1}\phi(1-0). \]

A simple computation shows that the eigenvalues of the above are given by 

\[ \frac{ 1 - bc}{a} \pm \frac{2\,i\,\sqrt{bc}\sin \beta \pi}{a} \]

with $ a = e^{-2i\beta \pi} - bc.$ Thus the ``new'' $\beta$s that determine the asymptotics are given by 
\[ \frac{ \log (\frac{ 1 - bc}{a} \pm \frac{2\,i\,\sqrt{bc}\sin \beta \pi}{a} )}{2 \pi i} .\]

The point of this example is that one cannot simply read off the asymptotics without doing the appropriate linear algebra. 
\end{example}

\begin{example}
Another example of a matrix-valued discontinuous symbol can be found in \cite{IJK, IMM}, which is related to entanglement in quantum spin chains. As in \eqref{e:entropy-symb}, consider the symbol
\begin{equation}\label{ex3-symb}
	\phi(e^{i\theta}) = \twotwo{i\lambda}{g(\theta)}{-g(\theta)^{-1}}{i\lambda}
\end{equation}
where 
\[ g(\theta) = \frac{\alpha \cos \theta -1 - i \gamma \alpha \sin \theta}{|\alpha \cos \theta -1 - i \gamma \alpha \sin \theta|} .\]
When $\alpha = 1$ this symbol has a jump at $\theta=0$. 
We will also assume that the parameter $\gamma$ is positive. If we compute the appropriate ``jump ratio matrix'' we find that it is
\[\frac{1}{ \lambda^{2} - 1} \twotwo{\lambda^{2} +1}{2\lambda}{2\lambda}{\lambda^{2} +1}.\] 
The eigenvalues of the above are 
\[ 
\frac{\lambda + 1}{\lambda - 1}  \quad\mbox{and}\quad \frac{\lambda - 1}{\lambda + 1} ,
\] 
and thus when $\lambda$ is real and  in the interval $( -1 ,1)$ the corresponding $\beta$s have real parts $\pm 1/2.$ This is a situation not covered by our theorem, although it is possible the results still hold. At $\lambda  = \pm1,$ the symbol is not invertible and thus also not covered by our theorem. For other values of $\lambda$ the asymptotics of the determinants are covered by our results. Note that $\det \phi=1-\lambda^2$ is a constant, and it is readily verified that the $I$-winding number of $\phi$ is zero. We have now obtained the following result:

\begin{theorem}
Let $\phi$ be given by~\eqref{ex3-symb} and suppose that $\lambda\notin [-1,1]$. Then
$$
	D_n[\phi] \sim (1-\lambda^{2})^{n} \,\,n^{\Omega}\,\, E
$$
where $\Omega = -2\beta^{2},$  $E$ is given in \eqref{f.constE1}, and $\beta = \frac{1}{2\pi i} \log (\frac{\lambda + 1}{\lambda - 1})$ with the appropriately chosen logarithms. 
\end{theorem}

Since in applications, one takes the logarithm of the above expression and then integrates, it is useful to know that the constant $E$ is not zero. This is not difficult to check in this case. We need to know that the operators $T(\phi)$ and $T(\tilde{\phi})$ are both invertible. Consider first $T(\phi).$

Notice the symbol is of the form 
\[ \phi= i\lambda I + \psi  = i(\lambda I - \psi),\quad
\psi=\twotwo{0}{ig(\theta)}{-ig(\theta)^{-1}}{0}. \] 
Because $|g(\theta)|^2=1$ we can conclude that the operator $T(\psi)$ is self-adjoint, while at the same time it has norm equal to one. Therefore the spectrum of  $T(\psi)$ is contained in the interval $[-1,1]$. 
Thus for $\lambda$ not in this interval, the operator $T(\phi)$ is invertible. The same argument holds for 
$T(\tilde{\phi})$ and thus the constant $E$ does not vanish. 
\end{example}

\begin{example}
Another example where the jump discontinuities occur can be found in \cite{ AEFQ1, AEFQ2} where similar entanglement problems are also discussed. There the symbol in question is of the form 
\begin{equation}\label{e:Renyi-symb}
\phi(e^{i\theta}) = \lambda I - \frac{1}{\Lambda (\theta)} M(\theta),
\end{equation}
where 
\[M(\theta) = \twotwo{ h + 2 \cos \theta} {G(\theta)} {- G(\theta)}{ -h - 2 \cos \theta}, \,\, \Lambda(\theta) = \sqrt{(h + 2\cos \theta)^{2} + |G(\theta)|^2}\] 
\[ G( \theta) = \begin{cases}
   -i(\pi + \theta),   &   -\pi \leq \theta < -\theta_0 \\
    -i\theta,  &   -\theta_0 < \theta < \theta_0 \\
    i(\pi - \theta),   &   \theta_0 < \theta \leq \pi 
\end{cases} \]
and $h\in\R$, $h\neq \pm 2$ and $\lambda\in \Gamma$ (see~\eqref{e:Renyi}) are certain parameters. While not rigorously proving the asymptotics, the authors do compute the eigenvalues of the jump ratio matrix to find that the eigenvalues at both jumps (i.e., at $\theta=\pm\theta_0$) are given by 
\begin{equation}
\iota_\pm= \left( \frac{\sqrt{\lambda^{2} -  \cos^{2}(\Delta \xi/2)} \pm  \sin (\Delta \xi/2)}{\sqrt{ \lambda^{2} - 1}} \right)^{2}
\end{equation}
where $\Delta \xi  = \xi^{+} - \xi^{-} $ and 
\[ \cos \xi^{+} = \frac{h + 2 \cos \theta_{0}}{\sqrt{(h +2 \cos \theta_{0})^{2} + (\theta_0 - \pi)^{2}}}\] 
\[ \sin \xi^{+} =  \frac{\theta_0 - \pi}{\sqrt{(h + 2 \cos \theta_{0})^{2} + (\theta_0 - \pi)^{2}}}\] 
and
\[ \cos \xi^{-} =    \frac{h + 2 \cos \theta_{0}}{\sqrt{(h + 2 \cos \theta_{0})^{2} + \theta_0^{2}}
}\] 
\[ \sin \xi^{-} =  \frac{ \theta_0}{\sqrt{(h +2 \cos \theta_{0})^{2} + \theta_0^{2}}
}.\] 
Our results are in agreement with these. To see where our theorem applies, we note once again that the determinant of the symbol $\phi$ is constant in $\lambda$. In addition, it is also straightforward to see that the two eigenvalues $\iota_\pm$ for each point of discontinuity are algebraic inverses of each other. Thus the only difficulty can occur when the logarithms of the eigenvalues do not satisfy the appropriate condition, that is, when 
\begin{equation}\label{e:intervals}
	\lambda \notin [-1, -|\cos(\Delta \xi/2)|]\cup[|\cos(\Delta \xi/2)|,1].
\end{equation}

We have now derived the following result, which was stated in Section~IV of \cite{AEFQ2} and then used to compute the entanglement entropy for Kitaev chains with long-range pairing.

\begin{theorem}
Let $\phi$ be given by \eqref{e:Renyi-symb} and suppose that~\eqref{e:intervals} holds. Then
$$
	D_n[\phi] \sim (\lambda^2-1)^n \, n^\Omega\, E,
$$
where $\Omega=-4\beta^2$ and the logarithm $\beta=\pm \frac{1}{2\pi i}\log \iota_\pm$ is appropriately chosen.
\end{theorem}

The operator for this theorem is of the form $\lambda I - T(\psi)$ with $T(\psi)$ self-adjoint and norm one. Hence the same argument applies as before, for $\lambda$ not in $[-1,1],$ the relevant operators are invertible and the constant $E$ is nonzero. We have not yet determined whether $E$ is nonzero in the case $\lambda$ is in the interval $$(-|\cos(\Delta \xi/2)|, |\cos(\Delta \xi/2)|).$$
\end{example}

\begin{example}
The final class of examples that we consider consists of piecewise constant matrices. We let 
\[ \psi(\theta) = \Big\{\begin{array}{cc}
I &  0 < \theta < \phi \\
M & \phi < \theta < 2 \pi
\end{array} \]
where $M$ is a constant invertible matrix. 
Suppose $M = \twotwo{a-1}{b}{c}{d-1}.$ Then  we can write our symbol as 
\[
 \twotwo{1}{0}{0}{1} + \chi_{(\phi, 2\pi)}\twotwo{a}{b}{c}{d} \]
where $\chi_{(\phi, 2\pi)}$ is the indicator function of the interval $[\phi, 2\pi].$ 
Then it is known \cite[Theorem~4.12]{LS} that the block matrix factors as
\[\frac{-1}{2b\nu} \twotwo{b}{b}{-(a-d)/2 + \nu}{-(a-d)/2 - \nu}\twotwo{\lambda_{1}}{0}{0}{\lambda_{2}}\twotwo{-(a-d)/2 - \nu}{-b}{(a-d)/2 - \nu}{b}\]

where $\nu = \sqrt{ cb + \frac{1}{4}(a-d)^{2}}$, and the functions 
\[ \lambda_{1} = 1 + ((a+d)/2 + \nu )\chi_{(\phi, 2\pi)} \]
and
\[ \lambda_{2} = 1 + ((a+d)/2 - \nu )\chi_{(\phi, 2\pi)}. \]
Thus the asymptotics reduce to the scalar Fisher-Hartwig case and are completely describable. We should note that the above holds if $b$ and $\nu$ are not zero, and if so, other factorizations are described by the referenced theorem. 
\end{example}

\medskip

\noindent{\bf Acknowledgments.} This work was supported in part by the
American Institute of Mathematics SQuaRE program ``Asymptotic behavior
of Toeplitz and Toeplitz + Hankel determinants,'' which ran from 2018 to 2022 and funded three meetings at AIM in San Jose, CA, and one meeting online. Virtanen was also supported in part by Engineering and Physical Sciences Research Council (EPSRC) grants EP/T008636/1, EP/X024555/1, and EP/Y008375/1.


\begin{thebibliography}{XXXX}

\bibitem[Arens]{Arens}
{\sc R. Arens}, 
To what extent does the space of maximal ideals determine the algebras?, 
in: Function algebras, Scott-Foresman (Chicago, Ill., 1966), 164--168.

\bibitem[AEFQ1]{AEFQ1}
{\sc F. Ares, J.G Esteve, F. Falceto, A.R. de Queiroz,}
Entanglement in fermionic chains with finite-range coupling and broken symmetries,
{\em Phys. Review A} 92, 0642334 (2015).

\bibitem[AEFQ2]{AEFQ2}
{\sc F. Ares, J.G Esteve, F. Falceto, A.R. de Queiroz,}
Entanglement entropy in the long-range Kitaev chain,
{\em Phys. Review A} 97, 062301 (2018).

\bibitem[BaTs]{BaTs}
{\sc H. Bart, V. E. Tsekanovskii,}
Matricial coupling and equivalence after extension,
Oper.\ Theory Adv.\ Appl., 59, Birkh\"auser, Basel, 1992, p. 143--160.

\bibitem[B1978]{B1978} 
{\sc E. L. Basor,}
Asymptotic formulas for Toeplitz determinants,
{\em Trans.\ Amer.\ Math.\ Soc.}\ {\bf  239} (1978), 33--65.

\bibitem[B1979]{B1979} 
{\sc E. L. Basor,}
A localization theorem for Toeplitz determinants,
{\em Indiana Univ. Math. J.} {\bf  28} (1979), no. 6, 975--983. 

\bibitem[BM]{BM}
{\sc E.L. Basor, K.E. Morrison},
The extended Fisher-Hartwig conjecture for symbols with multiple jump discontinuities,
in: Oper.\ Theory Adv.\ Appl., 71, Birkh\"auser, Basel, 1994, p.16--28.

\bibitem[Blek]{Blek}
{\sc P.M. Blekher}, 
The Fisher-Hartwig conjecture in the theory of Toeplitz matrices,
{\em Func.~ Anal. Appl.}, {\bf 16}, 79-83 (1982).

\bibitem[Bo82]{Bo82}
{\sc A. B\"ottcher}, 
Toeplitz determinants with piecewise continuous generating functions,
{\em Z. Anal. Anwendungen} {\bf 1}, no.~2, 23-39 (1982).

\bibitem[BS99]{BS99}
{\sc A. B\"ottcher, B. Silbermann},
{\em Introduction to large truncated Toeplitz matrices},
Universitext. Springer-Verlag, New York, 1999.

\bibitem[BS06]{BS06}
{\sc A. B\"ottcher, B. Silbermann},
{\em Analysis of Toeplitz operators}, 2nd edition. Prepared jointly with A.~Karlovich.
Springer Monographs in Mathematics, Springer, Berlin, 2006.

\bibitem[GIKMV]{GIKMV} 
{\sc L. Brightmore, G. Geh\'{e}r, A. Its, V. Korepin, F. Mezzadri, M. Mo, J. Virtanen,} 
Entanglement entropy of two disjoint intervals separated by one spin in a chain of free fermion, 
{\em J. Phys. A} 53 (2020), no. 34, 345303, 31 pp.

\bibitem[CG]{CG}
{\sc  K.F. Clancey, I. Gohberg}, 
{\em Factorization of matrix functions and singular integral operators,}
Operator Theory: Advances and Applications, Vol. 3, 
Birkh\"auser, Basel and Boston, 1981.

\bibitem[DIK11]{DIK11} 
{\sc P. Deift, A. R. Its, I. Krasovsky},
Asymptotics of Toeplitz, Hankel, and Toeplitz+Hankel determinants with Fisher-Hartwig singularities,
{\em Ann. of Math.} (2) {\bf 174} (2011), no. 2, 1243--1299.

\bibitem[DIK13]{DIK13}
{\sc P. Deift, A. R. Its, I. Krasovsky},
Toeplitz matrices and Toeplitz determinants under the impetus of the Ising model: some history and some recent results,
{\em  Comm. Pure Appl. Math.} {\bf  66} (2013), no. 9, 1360--1438.

\bibitem[E01]{E01} 
{\sc T. Ehrhardt}, 
A status report on the asymptotic behavior of Toeplitz determinants with Fisher-Hartwig singularities,
In: Recent advances in operator theory (Groningen, 1998), 
{\em Oper. Theory Adv. Appl.,} 124, Birkh\"auser, Basel, 2001, 217--241.

\bibitem[E03]{E03}
{\sc T. Ehrhardt},
A new algebraic approach to the Szeg\H{o}-Widom limit theorem.
{\em  Acta Math. Hungar.} {\bf  99} (2003), no. 3, 233--261. 

\bibitem[ES]{ES} {\sc T. Ehrhardt, I. M. Spitkovsky,}
On the kernel and cokernel of some Toeplitz operators. Advances in structured operator theory and related areas, 127--144,
Oper. Theory Adv. Appl., 237, Birkh\"auser/Springer, Basel, 2013.

\bibitem[GF]{GF}
{\sc I.C. Gohberg, I.A. Fel'dman,}
{\em Convolution equations and projection methods for their solution.} 
Translations of Mathematical Monographs, Vol. 41. American Mathematical Society, Providence, R.I., 1974. 

\bibitem[GGK]{GGK}
{\sc I. Gohberg, S. Goldberg, M.A. Kaashoek,}
{\em Classes of linear operators. Vol. II.}
Operator Theory: Advances and Applications, Vol.~63, Birkh\"auser, 
Basel, 1993. 

\bibitem[GKS]{GKS}
{\sc I. Gohberg, M.A. Kaashoek, I.M. Spitkovsky,}
An overview of matrix factorization theory and operator applications, in: 
Factorization and integrable systems (Faro, 2000), 
{\em Operator Theory: Advances and Applications,} vol. 141, 
Birkh\"auser Verlag, Basel and Boston, 2003, pp. 1--102.

\bibitem[LS]{LS}
{\sc G.S. Litvinchuk, I.M. Spitkovskii}, 
{\em Factorization of measurable matrix functions,}
Operator Theory: Advances and Applications, vol. 25, 
Birkh\"auser Verlag, Basel, 1987.

\bibitem[IJK]{IJK}
{\sc A. R. Its, B. Q. Jin and V. E. Korepin},
 Entanglement in the XY spin chain. 
 {\em J. Phys. A} {\bf 38} (2005), 2975-2990.

\bibitem[IMM]{IMM}
{\sc A.R. Its, F. Mezzadri, M.Y. Mo},
Entanglement entropy in quantum spin chains with finite range interaction. 
{\em Comm. Math. Phys.} {\bf 284} (2008), no.~1, 117--185. 

\bibitem[JK]{JK}  
{\sc B.-Q. Jin, V. E. Korepin,} 
Quantum spin chain, Toeplitz determinants and the Fisher-Hartwig conjecture. 
{\em J. Statist. Phys.} \textbf{116} (2004), no. 1-4, 79--95.

\bibitem[Ru]{Ru}
{\sc W. Rudin},
Functional Analysis, McGraw-Hill, Inc., New York, 1991.

\bibitem[Tay]{Tay}
{\sc J.L. Taylor}, Banach algebras and topology, 
in: {\em Algebras in analysis}, Academic Press
(London, 1975), 118--186.

\bibitem[V07]{V07} {\sc A. F. Voronin,} On the well-posedness of the Riemann boundary value problem with a matrix coefficient. (Russian)
Dokl. Akad. Nauk {\bf 414} (2007), no. 2, 156--158; translation in
Dokl. Math. {\bf 75} (2007), no. 3, 358--360.

\bibitem[Wi74]{Wi74} 
{\sc H. Widom},
Asymptotic behavior of block Toeplitz matrices and determinants.
{\em Advances in Math.}\ {\bf 13} (1974), 284--322. 

\bibitem[Wi75a]{Wi75a}
{\sc H. Widom},
Perturbing Fredholm operators to obtain invertible operators.
{\em J. Functional Analysis}\ {\bf  20} (1975), no.~1, 26--31. 

\bibitem[Wi75b]{Wi75b}
{\sc H. Widom},
On the limit of block Toeplitz determinants.
{\em Proc.~Amer.~Math.~Soc.}\ {\bf  50} (1975), 167--173. 

\bibitem[Wi76]{Wi76}
{\sc H. Widom},
Asymptotic behavior of block Toeplitz matrices and determinants. II.
{\em Advances in Math.}\ {\bf  21} (1976), no.~1, 1--29. 


\end{thebibliography}
\end{document}